\documentclass{siamart251216}

\usepackage{amsfonts,
			amsmath,
			graphicx,
			url}
\usepackage{algorithm,algpseudocode}
\usepackage{amssymb}
\usepackage{mathtools}
\usepackage{booktabs}
\usepackage{framed}
\usepackage{graphicx}
\usepackage{tikz}
\usepackage{epstopdf}
\usepackage{hyperref}
\usepackage{subfig}
\newcommand{\Prox}{\mathrm{Prox}}

\newcommand{\Proxrphi}{\mathrm{Prox}_{r\phi}}

\newcommand{\inner}[2]{\left\langle#1, #2\right\rangle}

\newcommand{\R}{\mathbb{R}}

\newcommand{\dx}{\textrm{d}x}

\newcommand{\xb}{\bar{x}}

\newtheorem{assumption}{Assumption}[section]
\newcommand{\kgrad}{\kappa_{\mathrm{grad}}}
\newcommand{\kobj}{\kappa_{\mathrm{obj}}}
\newcommand{\kfcd}{\kappa_{\mathrm{fcd}}}
\newcommand{\krad}{\kappa_{\mathrm{rad}}}

\newcommand{\ared}{\mathrm{ared}}
\newcommand{\pred}{\mathrm{pred}}
\newcommand{\cred}{\mathrm{cred}}
\newcommand{\dom}{\mathrm{dom }}

\DeclarePairedDelimiterX{\inp}[2]{\langle}{\rangle}{#1, #2}
\DeclarePairedDelimiterX{\norm}[1]{\|}{\|}{#1}

\usepackage{enumitem}
\newcommand{\crefpart}[2]{\Cref{#1}.\textup{\ref{#2}}}

\title{Convergence Analysis of an Inexact Proximal Trust-Region Method%
\thanks{%
This research was sponsored, 
in part, by the Department of Energy Office of Science under the Advanced 
Early Career Research Programs.  Sandia National Laboratories is a multimission laboratory managed and operated by
National Technology and Engineering Solutions of Sandia, LLC., a wholly owned subsidiary of Honeywell International, Inc., for the U.S.~Department of Energy's National Nuclear Security Administration under contract DE-NA0003525. This paper describes objective technical results and analysis. Any subjective views or opinions that might be expressed in the paper do not necessarily  represent the views of the U.S.~Department of Energy or the United States Government.     
}}              
            
\author{        
Robert J. Baraldi\footnotemark[2] \and                                                                                           
Drew P. Kouri\footnotemark[2] \and
Leandro Farias Maia\footnotemark[3]
}               
                
\ifpdf      
\hypersetup{
  pdftitle={Inexact Proximal Trust-Region Method},
  pdfauthor={Robert J. Baraldi, Drew P. Kouri, and Leandro F. Maia}                                                                                                                        
}               
\fi

\begin{document}

\maketitle
\renewcommand{\thefootnote}{\fnsymbol{footnote}}
\footnotetext[2]{Optimization and Uncertainty Quantification, Sandia National
  Laboratories, PO Box 5800, Albuquerque, 87185-1320, NM, USA;
  \{rjbaral@sandia.gov, dpkouri@sandia.gov\}}
\footnotetext[3]{Department of Mechanical, Industrial, and Manufacturing Engineering,
Oregon State University, Rogers Hall 410
Corvallis, OR 97331, USA; leandro.maia@oregonstate.edu}

\begin{abstract}
We analyze an inexact, proximal trust-region algorithm for minimizing the sum of a smooth nonconvex function and a nonsmooth convex function. 
Our algorithm leverages the proximal gradient as a benchmark to guarantee global convergence.  However, evaluating the proximal gradient involves solving a nonsmooth convex optimization problem that in many applications cannot be performed exactly.
As such, we develop a framework that permits inexact evaluations of the proximity operator and we show the proposed inexactness conditions yield descent properties analogous to the exact proximal gradient.
We further discuss practical procedures for computing approximate proximity operators for two common applications: {\em (i)} the nonsmooth term is the sum of nonsmooth convex functions where one is composed with an affine map and {\em (ii)} using alternative inner products to facilitate the evaluation of the proximity operator. 
We demonstrate our algorithm and confirm our analysis with two numerical experiments arising in PDE-constrained optimization.
\end{abstract}


\section{Introduction}
Our goal is the efficient numerical solution of the nonsmooth, nonconvex optimization problem
\begin{equation}\label{eq:optprob}
  \min_{x\in X}\; \left\{F(x) \coloneqq f(x)+\phi(x)\right\},
\end{equation}
where $X$ is a real Hilbert space, 
$f:X\to\R$ is a smooth nonconvex function and $\phi:X\to(-\infty,+\infty]$ is a proper, closed and convex function.
Nonsmooth problems of the form \eqref{eq:optprob} arise throughout modern optimization.
In particular, this structure appears naturally in sparse estimation, convex-constrained optimization, and inverse, control and design problems, see e.g., \cite{Baraldi2023,Beck17,kouri2018inexact} 
and the references therein. 
In the context of optimal control and design, the smooth term $f$ often depends on the state of a nonlinear physical system and the nonsmooth term $\phi$ encodes 
structural information such as sparsity, constraints, and variational penalties 
\cite{bendsoe2013topology,clason2018total,lazarov2011filters,manns2023on}. 
A central challenge in numerically solving \eqref{eq:optprob} is to exploit the problem structure without sacrificing the ability to handle nonconvexity, inexact objective function and derivative evaluations, and large-scale computation.  

The authors of \cite{Baraldi2023} propose a trust-region method for solving \eqref{eq:optprob} that exploits inexact evaluations of the smooth objective function $f$ and its gradient.
However, the convergence analysis for this trust-region method requires that the proximity operator of $\phi$ is computed exactly --- in practice this implies that the proximity operator has an analytical solution.
Since these closed-form solutions do not exist for many functions, we will generalize the method from \cite{Baraldi2023} to leverage inexact prox evaluations.

The proximity operator is a powerful tool for solving composite nonsmooth optimization problems like \eqref{eq:optprob}.
Algorithms that leverage the proximity operator, like ISTA and FISTA \cite{Beck17}, are often first-order and only apply to convex problems.
Like \cite{Baraldi2023}, these methods typically require that the proximity operator be computed exactly --- a requirement that is often computationally expensive or even impossible to satisfy in many applications and even simple $\phi$.
For example, if $X$ is finite dimensional with dot product defined by a positive diagonal matrix and $\phi$ is the $\ell^1$-norm, then the proximity operator has an analytical expression.
In contrast, if the dot product is defined by a non-diagonal symmetric positive definite matrix $M$, then evaluating the proximity operator requires a specialized iterative method to obtain an approximate solution.

First-order optimization methods that leverage inexact evaluations of the proximity operator have been studied in \cite{barre2023principled,MaiaInexactCyclic,Hua12,Kim-chuan12,phdMaia2025,salzo2012inexact,solodov1999hybrid,Richtarik14}.
In these works, the inexact proximity operator is characterized using approximate subdifferentials and fuzzy calculus, see e.g., \cite{clason2023introduction,mordukhovich2006variational} and the references therein.
For instance, \cite{salzo2012inexact} incorporates inexact evaluations of the proximity operator into 
momentum-based acceleration schemes for the proximal point method.
Most of these works only apply to convex problems and generally exhibit slow convergence rates. 
In contrast, the methods studied in
\cite{orban2025ir2n,aravkin2021proximal,orban2025indef} permit nonconvex $\phi$ and they allow $\phi$ to be evaluated inexactly, giving rise to inexact proximity operators.
However, these approaches can be restrictive to implement, requiring the model of $\phi$ to match in value and subdifferential at the center of the trust region.
The authors in \cite{orban2025ir2n} analyze a similar method for which they require that the norm of the computed step using the inexact proximity operator is at least a fraction of the minimum-norm element of the step computed using the exact (set-valued) one --- a condition that may be difficult to verify in practice unless a priori bounds are available on the exact step.
By extending the nonsmooth trust-region algorithm introduced in \cite{Baraldi2023}, 
we analyze a proximal Newton method that leverages inexact evaluations of the proximity operator, while maintaining the global convergence properties from \cite{Baraldi2023}.

The remainder of the paper is structured as follows.
\Cref{sec:prelim} introduces notation, definitions and problem assumptions used throughout, and presents a variety of inexact proximity operators based on the approximations introduced in \cite{salzo2012inexact}.
\Cref{sec:trust} provides a review of the proximal trust-region method from \cite{Baraldi2023} and the modifications that enable the use of inexact proximity operators.
In \Cref{sec:compprox}, we  demonstrate how to compute approximate proximity operators that satisfy the assumptions of our algorithm in some common situations, including when $\phi$ exhibits certain structure.
Finally, \Cref{sec:numerics} demonstrates the performance of our algorithm on two numerical examples from partial differential equation (PDE) constrained optimization.

\section{Preliminaries}
\label{sec:prelim}
Optimization in \eqref{eq:optprob} is performed over the real Hilbert space $X$ with inner product $\inp{\cdot}{\cdot}$ and norm $\|x\|=\inp{x}{x}^{1/2}$ for all $x\in X$.
Moreover, we require the following assumptions on the problem data.
\begin{assumption}[Problem data]
\label{as:probdata}
\begin{itemize}
    \item[1.] The function $\phi: X\to (-\infty,\infty]$ is proper, closed and convex with effective domain
\(
  \textup{dom}\,\phi\coloneqq \{x\in X\,\vert\, \phi(x)<+\infty\}.
\)
    \item[2.] The function $f : X \to \mathbb{R}$ is $L$-smooth on $\textup{dom }\phi$, i.e., there exists an open set $U \subseteq X$ with $\textup{dom }\phi\subset U$ on which $f$ is Fr\'echet differentiable with
    Lipschitz-continuous gradient $\nabla f(x)\in X$ and modulus $L>0$ for all $x\in U$.
    \item[3.] The objective function $F(x)\coloneqq f(x)+\phi(x)$ is bounded from below.
\end{itemize}
\end{assumption}

For notational convenience, we denote the closed unit ball in $X$ by $\mathbb{B}\coloneqq\{x\in X\,\vert\,\|x\|\le 1\}$ and to simplify the presentation, we identify the topological dual space $X^*$ with $X$ through the Riesz representation theorem. 
We note that all inner products, norms, etc.\ are with respect to $X$ unless otherwise indicated.
Finally, we denote the nonnegative and positive real numbers by $\mathbb{R}_{+}:=[0,+\infty)$ and $\mathbb{R}_{++}:=(0,+\infty)$, respectively.

In the subsequent sections, we will leverage various approximate subgradients to construct different approximate proximity operators.  We recall the standard definitions of the proximity operator and subdifferential here.
For a proper, closed and convex function $\phi:X\to(-\infty,+\infty]$, we denote the 
proximity operator of $\phi$ at $x\in X$ with index $r\in\mathbb{R}_{++}$ as the unique solution to the minimization problem
\begin{equation}\label{eq:exact.prox}
  \operatorname{Prox}_{r\phi}(x) \coloneqq \operatorname*{arg\,min}_{y\in X} \left\{\frac{1}{2r}\|y-x\|^2 + \phi(y)\right\},
\end{equation}
and the subdifferential of $\phi$ at $x\in X$ as the (possibly empty) set-valued map
\begin{equation}
\label{eq:subgradient.exact}
  \partial\phi(x) \coloneqq \{\xi\in X\,\vert\, \phi(y) \ge \phi(x) + \inp{\xi}{y-x} \quad\forall\, y\in X\},
\end{equation}
whose members $\xi\in\partial\phi(x)$ are called subgradients.
These two concepts are related through the nonsmooth Fermat's principle:
\begin{subequations}\label{eq:fermat-prox}
\begin{align}
  0\in \frac{1}{r}(p-x) + \partial\phi(p)
  &\iff
  x \in p + r \partial\phi(p) = (I+r\partial\phi)(p) \label{eq:fermat-prox-a} \\
  &\iff p=(I+r\partial\phi)^{-1}(x) \label{eq:fermat-prox-b}\\
  &\iff
  p = \operatorname{Prox}_{r\phi}(x), \label{eq:fermat-prox-c}
  \end{align}
\end{subequations}
where $I$ denotes the identity operator on $X$.
Note that \eqref{eq:fermat-prox-a} combined with the definition of the subdifferential \eqref{eq:subgradient.exact} yields
\begin{equation}\label{eq:subgrad-prox}
  \phi(y) \ge \phi(p) + \inp{r^{-1}(x-p)}{y-p} \quad\forall\, y\in X.
\end{equation}
Moreover, \eqref{eq:fermat-prox-b} indicates that the proximity operator of $\phi$ is the resolvent of the subdifferential $\partial\phi$ \cite[Equation~(24.6)]{bauschke2017monotone}.
These relationships are paramount in defining the various approximate proximity operators that we utilize in this work.
Additionally, the notion of the $\varepsilon$-subdifferential, $\varepsilon\in\mathbb{R}_+$, will also be critical.  That is, the $\varepsilon$-subdifferential of $\phi$ at $x\in X$ is the set-valued map
\begin{equation}\label{eq:subgradient.inexact}
  \partial_\varepsilon\phi(x) \coloneqq \{\xi\in X\,\vert\,\phi(y)\ge\phi(x)+\inp{\xi}{y-x} - \varepsilon\;\;\forall\, y\in X\},
\end{equation}
where members $\xi\in\partial_\varepsilon\phi(x)$ are called $\varepsilon$-subgradients of $\phi$ at $x$.

\subsection{Basic Properties of the Proximity Operator}
We now present several properties of the exact  proximal operator defined in~\eqref{eq:exact.prox}. 
For this and the subsequent discussions of the inexact proximity operator, we introduce the function $\phi_x^{(r)}:X\to (-\infty,+\infty]$ for fixed $x\in X$ and $r\in\mathbb{R}_{++}$ defined by
\begin{align}\label{eq:phi.x.r}
\phi_x^{(r)}(y) = \phi(y) + \frac{1}{2r}\|y-x\|^2 
\end{align}    
and recall that $\Proxrphi(x)$ is the unique minimizer of $\phi_x^{(r)}(y)$ over $y\in X$.
Using the definition of the subgradient, we can prove the following property about the regularized functional $\phi_x^{(r)}$. 
\begin{lemma}
\label{lem:prox-dec}
Let $x \in X$.
For every $y\in X$ and $r\in\mathbb{R}_{++}$, it holds that
\[
    \phi_x^{(r)}(y)-\phi_x^{(r)}(\Proxrphi(x)) \geq \frac{1}{2r}\norm{y-\Proxrphi(x)}^2.
\]
\end{lemma}
\begin{proof}
  Observe that $\phi_x^{(r)}$ is $(1/r)$-strongly convex. Hence, for every $z\in X$,
  \[
  \phi_x^{(r)}(y) \geq \phi_x^{(r)}(z) + \inner{\xi}{y-z} + \frac{1}{2r}\norm{y-z}^2,
  \]
for every $\xi\in \partial \phi_x^{(r)}(z)$. 
Since $\Proxrphi(x)$ minimizes $\phi_x^{(r)}(\cdot)$, $0\in \partial \phi_x^{(r)}(\Proxrphi(x))$. 
Choosing $z=\Proxrphi(x)$ and $\xi = 0$ concludes the proof.
\end{proof}
The next lemma demonstrates properties of the map  $r\mapsto \|\Proxrphi(x+rd)-x\|$.
\begin{lemma}\label{lem:prox.decrease}
    Let $x,d\in X$, $r\in\mathbb{R}_{++}$, and define the functions
\begin{equation}
\label{eq:psi}
  \varphi(r) =\|\Proxrphi(x+rd)-x\|
  \qquad\text{and}\qquad
  \psi(r) =\frac{1}{r}\varphi(r).
\end{equation}
Then, the mapping $r\mapsto \varphi(r)$ is nondecreasing and the mapping $r\mapsto\psi(r)$ is nonincreasing for $r>0$, i.e., if $r\geq t$ then $\varphi(r)\ge \varphi(t)$ and $\psi(r) \le \psi(t)$. 
Additionally, the inequalities are strict if $\Proxrphi(x+rd)\neq\Proxrphi(x+td)$.
\end{lemma}
\begin{proof}
    See~\cite[Lemma~2]{Baraldi2023}.
\end{proof}

The next result, sometimes called the {\em descent lemma}, shows that if $r\in(0,2/L)$, then the proximal-gradient step
$$\Prox_{r\phi}(x-r\nabla f(x))-x$$ yields descent for the objective function $F(x)=f(x)+\phi(x)$.
This descent property is fundamental for extending the classical Cauchy-point construction used in traditional trust-region methods --- which is used for guaranteeing sufficient decrease at each iteration --- to the nonsmooth setting of \eqref{eq:optprob}.

\begin{lemma}[Descent Lemma]\label{lem:pg_descent}
Let Assumption~\ref{as:probdata} hold and fix $x\in X$ and $r\in(0,2/L)$.  Then the proximal gradient iterate
\[
x^+ \coloneqq \Prox_{r\phi}\big(x-r\nabla f(x)\big),
\]
satisfies \(F(x^+) \le F(x),\) with strict inequality when $x^+\neq x$.
\end{lemma}
\begin{proof}
Since $f$ is $L$-smooth by Assumption~\ref{as:probdata}, it satisfies
\[
f (x^+) \leq f (x) + \inp{\nabla f(x)}{x^+-x} + \frac{L}{2}\|x^+-x\|^2.
\]
Moreover, by inequality~\eqref{eq:subgrad-prox}, we have
\[
\phi(x) \geq \phi(x^+) + \inner{\frac{x-r\nabla f(x)-x^+}{r}}{x-x^+}.
\]
Adding the two previous inequalities and simplifying yields
\begin{align*}
    F(x^+)\leq F(x) + \left(\frac{L}{2}-\frac{1}{r}\right)\|x-x^+\|^2.
\end{align*}
The result then follows from our choice of $r$, which ensures that $(L/2-1/r)<0$.
\end{proof}

\subsection{Inexact Proximity Operator}
\label{subsec:inexact.subgradient.proximal}
There are various notions of inexact proximity operators.
We follow \cite{salzo2012inexact} and define three approximations based on relaxations of the equivalent expressions in \eqref{eq:fermat-prox}.
\begin{definition}[Approximate Proximity Operator]
\label{def:type1-error}
Fix $x\in X$, $r\in\mathbb{R}_{++}$, $\varepsilon\in\mathbb{R}_+$ and define $\tau=\tau(\varepsilon)\coloneqq \varepsilon^2/(2r)$.
\begin{itemize}[align=left, widest={Type 3:}, leftmargin=*]
\item[\bf Type 1:] $u\in X$ is a Type-1 approximation of $\Proxrphi(x)$, denoted $u\approx_1\Proxrphi(x)$, with $\varepsilon$-precision if
  \begin{equation}\label{eq:type1-error}
    0\in \partial_{\tau}\phi_x^{(r)}(u)
  \end{equation}
\item[\bf Type 2:] $u\in X$ is a Type-2 approximation of $\Proxrphi(x)$, denoted $u\approx_2\Proxrphi(x)$, with $\varepsilon$-precision if
  \begin{equation}\label{eq:type2-error}
    \frac{1}{r}(x-u) \in \partial_{\tau}\phi(u)
  \end{equation}
\item[\bf Type 3:] $u\in X$ is a Type-3 approximation of $\Proxrphi(x)$, denoted $u\approx_3\Proxrphi(x)$, with $\varepsilon$-precision if
 \begin{equation}\label{eq:type3-error}
   \inf_{\xi\in\partial\phi_x^{(r)}(u)}\|\xi\| \le \frac{\varepsilon}{r},
    \end{equation}
    where by convention the infimum is $+\infty$ whenever $\partial\phi_x^{(r)}(u)=\emptyset$.
    Note that the infimum is attained when $\partial\phi_x^{(r)}(u)\neq\emptyset$ since $\partial\phi_x^{(r)}(u)$ is closed and convex, cf.\ \cite[Theorem~3.16~\&~Proposition~16.4]{bauschke2017monotone}.
\end{itemize}
\end{definition}
Notice that the Type-1, 2, and 3 approximations arise by relaxing  \eqref{eq:fermat-prox-a} in various ways.
To better understand the Type-2 approximation, we note that the inclusion \eqref{eq:type2-error} is equivalent to
\[
x\in u + r\partial_{\tau}\phi(u) \iff u \in (I+r\partial_{\tau}\phi)^{-1}(x)
\]
and therefore, Type-2 approximations are resolvents of $\partial_\tau\phi$, which are set valued in general.
This is analogous to the relationship \eqref{eq:fermat-prox-b}.
Moreover, the following sequence of equivalences are particularly useful for understanding Type-3 approximations.
By definition, we have that $u\approx_3\Proxrphi(x)$ with $\varepsilon$-precision if and only if $\partial\phi_x^{(r)}(u)\neq\emptyset$ and
\begin{align*}
  \exists\, e\in \frac{\varepsilon}{r}\mathbb{B}\cap \partial\phi_x^{(r)}(u) &\iff \exists\, e\in \varepsilon\mathbb{B}\quad\text{such that}\quad(x+e-u)/r\in\partial\phi(u) \\
    &\iff \exists\,e\in \varepsilon\mathbb{B}\quad\text{such that}\quad u = \Proxrphi(x+e).
\end{align*}
Consequently, Type-2 approximations arise from enlarging $\partial\phi$, while Type-3 approximations arise from perturbing the argument $x$.

In the subsequent sections, we prove convergence of our trust-region algorithm when using Type-1 approximations because they are the most general of the three.
In particular, the next result demonstrates that both Type-2 and 3 approximations yield Type-1 approximations. 
\begin{proposition}\label{prop:salzo.error.equival}
    Let $x\in X$ and $r,\,\varepsilon\in\R_{++}$. The following hold:
    \begin{enumerate}
        \item If $u\approx_2\Proxrphi(x)$ with $\varepsilon$-precision, then $u\approx_1\Proxrphi(x)$ with $\varepsilon$-precision.
        \item If $u\approx_3\Proxrphi(x)$ with $\varepsilon$-precision, then $u\approx_1\Proxrphi(x)$ with $\varepsilon$-precision.
    \end{enumerate}
\end{proposition}
\begin{proof}
See \cite[Proposition~1]{salzo2012inexact}.
\end{proof}

In addition to the connections between Type-1, 2 and 3, we can quantify the error for the Type-1 approximation as done in the following lemma.
\begin{lemma}\label{lem:inexac.exact.prox.diff}
    For a given $x\in X$ and $r\in \R_{++}$, if $u\approx_1 \Proxrphi(x)$ with $\varepsilon$-precision, then
    \begin{equation}
    \|u-\Proxrphi(x)\| \le \varepsilon
    \end{equation}
\end{lemma}
\begin{proof}
    The inclusion \eqref{eq:type1-error} ensures that
    \[
    \phi_x^{(r)}(u) \leq \phi_x^{(r)}(\Proxrphi(x)) + \varepsilon^2/(2r),
    \]
    which when combined with \Cref{lem:prox-dec} yields the desired bound.
\end{proof}

To build on the intuition behind Type-3 approximations, we present a connection with another notion of approximate subgradient called the $\delta$-Fr\'{e}chet subgradient.
One reason for doing this is that $\delta$-Fr\'{e}chet subgradients are often simpler to compute in practice than classical approximate subgradients.
\begin{definition}[$\delta$-Fr{\'e}chet Subgradient]
\label{def:delta.Frechet}
For $x\in \dom\; \phi$ and $\delta\in\R_{+}$, the $\delta$-Fr\'echet subdifferential of $\phi$ at $x$ is
\[
\hat{\partial}_\delta\phi(x)\coloneqq \{ \xi\in X\,\vert\,\phi(y) \geq \phi(x) + \inner{\xi}{y-x} - \delta \|y-x\| \;\; \forall y\in X\}.
\]
According to \cite[Proposition~1.34]{kruger2004frechet}, the $\delta$-Fr\'{e}chet subdifferential satisfies
\[
  \hat{\partial}_\delta\phi(x) = \partial[\phi+\delta\|\cdot-x\|](x) = \partial\phi(x)+\delta\mathbb{B}.
\]
\end{definition}

In the next result, we demonstrate the connection between the $\delta$-Fr\'echet subdifferential and Type-3 approximation.
\begin{lemma}\label{lem:delta-prox-char}
  Let $x,\, u\in X$, $\delta\in\R_+$ and $r\in\R_{++}$, then the following statements are equivalent:
  \begin{enumerate}\itemsep2pt
    \item $0\in\hat{\partial}_{\delta}\phi_x^{(r)}(u)$;\label{lem:delta-prox-char1}
    \item $\inf_{\xi\in\partial\phi_x^{(r)}(u)}\|\xi\|\le\delta$;\label{lem:delta-prox-char2}
    \item $\exists\,\xi\in \delta\mathbb{B}$ such that $(x-u)/r + \xi\in\partial\phi(u)$;\label{lem:delta-prox-char3}
    \item $\exists\,\xi\in\delta\mathbb{B}$ such that $u=\Proxrphi(x+r\xi)$. \label{lem:delta-prox-char4}
  \end{enumerate}
\end{lemma}
\begin{proof}
  We first note that \crefpart{lem:delta-prox-char}{lem:delta-prox-char3} and \crefpart{lem:delta-prox-char}{lem:delta-prox-char4} are clearly equivalent since
  \[
    (x-u)/r+\xi\in\partial\phi(u)
    \iff x+r\xi\in (I+r\partial\phi)(u)\iff u=\Proxrphi(x+r\xi).
  \]
  Now, the equivalence between \crefpart{lem:delta-prox-char}{lem:delta-prox-char3} and \crefpart{lem:delta-prox-char}{lem:delta-prox-char1} follows from the fact that
  \[
   \hat{\partial}_\delta \phi_x^{(r)}(u) = (u-x)/r + \hat{\partial}_\delta \phi(u) = (u-x)/r + \partial\phi(u) + \delta\mathbb{B}
  \]
\cite[Lemma~17.1~\&~Corollary~17.3]{clason2023introduction}.
  Finally, to prove equivalence of \crefpart{lem:delta-prox-char}{lem:delta-prox-char3} and \crefpart{lem:delta-prox-char}{lem:delta-prox-char2}, 
  we note that \crefpart{lem:delta-prox-char}{lem:delta-prox-char3} is equivalent to
  \[
    \exists\, \xi_1\in\partial\phi(u),\;\;\xi_2\in\delta\mathbb{B} \quad\text{such that}\quad
    (x-u)/r = \xi_1-\xi_2.
  \]
  Rearranging terms yields
  \[
    \xi_1 = (x-u)/r + \xi_2\in\partial\phi(u).
  \]
  Consequently,
  \[
   \inf_{\xi\in\partial\phi_x^{(r)}(u)}\|\xi\| = \inf_{\xi\in\partial\phi(u)}\|(u-x)/r+\xi\|
   \le \|(u-x)/r+\xi_1\| = \|\xi_2\| \le \delta,
  \]
  proving the implication \crefpart{lem:delta-prox-char}{lem:delta-prox-char1} $\implies$ \crefpart{lem:delta-prox-char}{lem:delta-prox-char2}. 
  To prove the reverse implication, \crefpart{lem:delta-prox-char}{lem:delta-prox-char2} is equivalent to 
  \[
    \exists\, \xi\in\partial\phi(u) \qquad\text{such that}\qquad
    \|(x-u)/r - \xi\| \le \delta,
  \]
  and so
  \[
    (x-u)/r = [(x-u)/r - \xi] + \xi \in \partial\phi(u) + \delta\mathbb{B} = \hat{\partial}_\delta\phi(u)
  \]
  as desired.
\end{proof}
It is clear that the $u$ in \Cref{lem:delta-prox-char} satisfies $u\approx_3 \Proxrphi(x)$ with $\varepsilon$-precision if $\delta=\varepsilon/r$.
Moreover, \Cref{lem:delta-prox-char} demonstrates that Type-3 approximations are Type-1 and 2 approximations with the classical $\tau$-subdifferentials $\partial_\tau\phi_x^{(r)}(u)$ and $\partial_\tau\phi(u)$ replaced by the $\delta$-Fr\'{e}chet subdifferential.
Our final result of this section plays an important role in the global convergence analysis of our trust-region method.
In particular, this result allows us to prove descent of the inexact proximal gradient step.
\begin{lemma}\label{lem:decr.technical}
Let $x,g\in X$, $r\in\R_{++}$ and $\varepsilon\in\R_+$, and suppose $u\approx_1\Proxrphi(x-rg)$ with $\varepsilon$-precision.
\begin{enumerate}
  \item It holds that
  \begin{equation}\label{lem:decr.technical.a}
    \inner{g}{u-x}+\phi(u) - \phi(x) \leq -\frac{1}{2r}\|u-x\|^2 + \frac{\varepsilon^2}{2r}.
\end{equation}
\label{lem:decr.technical.1}
\item Let $\kappa\in(0,\tfrac12)$ be fixed and suppose $x\neq\Proxrphi(x-rg)$.  
If
\begin{equation}\label{lem:decr.technical.a.b1}
  0 < \varepsilon \le \kappa\|\Proxrphi(x-rg)-x\|,
\end{equation}
then
\[
\inner{g}{u-x}+\phi(u) - \phi(x) \le -\frac{1-2\kappa}{2r(1-\kappa)^2}\|u-x\|^2.
\]
\label{lem:decr.technical.2}
\end{enumerate}
\end{lemma}
\begin{proof}
Let $\bar x=x-rg$.
To prove \crefpart{lem:decr.technical}{lem:decr.technical.1}, it follows from \eqref{eq:type1-error} that 
    \[
    \frac{1}{2r}\|u-\bar x\|^2 + \phi(u) \leq \frac{1}{2r}\|x-\bar x\|^2 + \phi(x) + \frac{\varepsilon^2}{2r}.
    \]
    Upon expanding the quadratic terms on the left-hand side and rearranging the previous inequality, we see that \eqref{lem:decr.technical.a} holds.
For \crefpart{lem:decr.technical}{lem:decr.technical.2}, the triangle inequality and \Cref{lem:inexac.exact.prox.diff} yield
\begin{equation}\label{proof.Q.decrease.1}
\|\Proxrphi(\bar x)-x\| \le \|\Proxrphi(\bar x)-u\| + \|u-x\| \le \varepsilon + \|u-x\|,
\end{equation}
which, when combined with the bound for $\varepsilon$ in \eqref{lem:decr.technical.a.b1}, implies
    \(
    \varepsilon \leq \left(\frac{\kappa}{1-\kappa}\right)\|u-x\|.
    \)
    Combining this bound with \eqref{lem:decr.technical.a}, we obtain
    \begin{align*}
        \inner{g}{u-x}+\phi(u) - \phi(x) \leq -\frac{1}{2r}\|u-x\|^2 + \frac{\varepsilon^2}{2r} \le -\frac{1-2\kappa}{2r(1-\kappa)^2}\|u-x\|^2,
    \end{align*}
    which concludes the proof.
\end{proof}
Notice that \Cref{lem:decr.technical} enables us to prove an analogous result as \Cref{lem:pg_descent} for Type-1 proximal gradient approximations.
In particular, for $u$ as in \Cref{lem:decr.technical} with $g=\nabla f(x)$, if \eqref{lem:decr.technical.a.b1} holds and $0 < r < \frac{1-2\kappa}{L(1-\kappa)^2}$, then $F(u) < F(x)$ when $u\neq x$.
In the subsequent sections, we will leverage this type of result to prove that the benchmark Cauchy point, along the inexact proximal gradient path, produces sufficient decrease, therefore ensuring that our trust-region method is globally convergent as in \cite{Baraldi2023}.

\section{A Proximal Trust{-}Region Method}
\label{sec:trust}

In this section, we describe \cite[Algorithm~1]{Baraldi2023} and present modifications, where needed, to account for inexact evaluations of the proximity operator.
Let $x_k$ denote the $k$-th iteration of the trust-region method and $g_k$ an approximation of the gradient $\nabla f(x_k)$.
There are two key quantities, where inexact proximity operator evaluations play a critical role: {\em (i)} the computation of the Cauchy point as a point along the proximal gradient path
\begin{equation}\label{eq:path}
  p_k(r)\coloneqq x_k(r)-x_k \quad\text{for} \quad r\in\R_{++},
\end{equation}
where $x_k(r)\approx_1\Proxrphi(x_k-rg_k)$ with $\varepsilon_k$-precision and {\em (ii)} the evaluation of the proximal gradient stationarity metric
\begin{equation}
    \label{eq:psik}
    \psi_k(r)\coloneqq \frac{1}{r}\|p_k(r)\| \qquad\text{and}\qquad
    \Psi_k \coloneqq \psi_k(r_k),
\end{equation}
where $r_k\in\R_{++}$ is the Cauchy point step length.
For notational convenience, we denote $\bar{x}_k(r)=\Proxrphi(x_k-rg_k)$ and define $\bar{p}_k(r)$, $\bar{\psi}_k(r)$ and $\bar{\Psi}_k$ analogously to $p_k(r)$, $\psi_k(r)$and $\Psi_k$.
Recall that if there exists $\bar{r}\in\R_{++}$ for which $x_k=\bar{x}_k(\bar{r})$, then $x_k=\bar{x}_k(r)$ for all $r\in\R_{++}$, owing to the monotonicity properties of $r\mapsto r\bar\psi_k(r)$ and $r\mapsto \bar\psi_k(r)$ \cite[Lemma~2]{Baraldi2023}.

Trust{-}region methods are iterative procedures for computing approximate solutions to general nonconvex optimization problems \cite{conn2000trust} and are ideal for leveraging and exploiting inexact computations \cite{Baraldi2023,kouri2018inexact}.
In the recent work \cite{Baraldi2023}, the authors developed a proximal trust-region method for nonsmooth optimization problems with the form \eqref{eq:optprob} that rigorously handles inexact evaluations of $F$ and the gradient $\nabla f$ with guaranteed convergence.
At the $k$-th iteration of \cite[Algorithm~1]{Baraldi2023}, we compute a trial iterate $x_k^+$ that {\em approximately solves} the trust-region subproblem
\begin{equation}\label{eq:tr-sub}
  \min_{x\in X}\{m_k(x)\coloneqq f_k(x)+\phi(x)\} \qquad \text{subject to}\qquad \|x-x_k\|\le \Delta_k,
\end{equation}
where $x_k \in \textup{dom}\,\phi$ is the current iterate, $f_k$ is a smooth local model of $f$ around $x_k$, and $\Delta_k > 0$ is the trust-region radius.
Although general nonlinear models $f_k$ are possible, it is often numerically convenient to employ the quadratic model
\[
  f_k(x) = \frac{1}{2}\inp{x-x_k}{B_k(x-x_k)}+\inp{g_k}{x-x_k},
\]
where $B_k:X\to X$ is a continuous self-adjoint linear operator and $g_k\approx\nabla f(x_k)$.
As such, we will restrict our attention to quadratic $f_k$.
When saying that $x_k^+$ {\em approximately solves} \eqref{eq:tr-sub}, we mean that the trial iterate $x_k^+$ is feasible with respect to the trust-region constraint up to a scaling of the trust-region radius and satisfies the fraction of Cauchy decrease condition. That is, there exist positive constants $\krad$ and $\kfcd$, independent of $k$, such that
\begin{subequations}\label{eq:FCD.condition}
\begin{align}
  \|x_k^+-x_k\| &\le \krad \Delta_k \label{eq:FCD.condition.1}\\
  m_k(x_k)-m_{k}(x_k^+) &\ge \kfcd \Psi_k\min\left\{\frac{\Psi_k}{1+\|B_k\|},\Delta_k\right\} \label{eq:FCD.condition.2},
\end{align}
\end{subequations}
where $\|B_k\|$ denotes the operator norm of $B_k$.
See \cite{Baraldi2023} for the convergence analysis when more general models $f_k$ are used.

Given a trial iterate $x_k^+$ that satisfies \eqref{eq:FCD.condition}, 
we decide whether to accept or reject $x_k^+$ based on the ratio of computed reduction 
\[
  \cred_k\approx\ared_k\coloneqq F(x_k)-F(x_k^+)
\]
and the predicted reduction
\[
  \pred_k\coloneqq m_k(x_k)-m_k(x_k^+).
\]
In particular, if
\begin{align}\label{eq:rho}
  \rho_k\coloneqq\frac{\cred_k}{\pred_k} \ge \eta_1,
\end{align}
we accept $x_k^+$ for user specified $\eta_1\in(0,1)$, i.e., $x_{k+1}=x_k^+$. Otherwise, we reject it, setting $x_{k+1}=x_k$. We then use $\rho_k$ to increase or decrease the trust-region radius $\Delta_k$.

To complete the description of the inexact proximal trust-region method, we require the following assumptions on the inexact evaluations of the objective function $F$, the gradient $\nabla f$, and the stationarity metric $\Psi_k$. 

\begin{assumption}[Inexact Evaluations]
\label{assump:inexact}
For user-specified $r_0\in\R_{++}$, independent of $k$, the following error bounds hold.
\begin{enumerate}
\item {\bf Inexact Value:} There exists $\kobj\in\R_+$, independent of $k$, such that
    \begin{equation}\label{eq:assump.inexact.objective}
        |\ared_k-\cred_k| \leq \kobj\left[ \eta_{\textup{obj}}\min\{\pred_k,\theta_k\}\right]^{\zeta_{\textup{obj}}}
    \end{equation}
where the user-specified positive parameters $\zeta_{\textup{obj}}$, $\eta_{\textup{obj}}$ and $\theta_k$ satisfy
\[
\zeta_{\textup{obj}}>1,\quad 0<\eta_{\textup{obj}}<\min\{\eta_1,(1-\eta_2)\}\quad \text{and}\quad \lim_{k\to+\infty}\theta_k=0.
\]
\label{assump:inexact.object}
\item {\bf Inexact Gradient:} There exists $\kgrad\in\R_+$, independent of $k$, such that
\begin{equation}\label{eq:asssump.inexact.gradeint}
\|g_k - \nabla f(x_k)\|\le \kgrad\min \{\psi_k(r_0),\Delta_k\}.
\end{equation}
\label{assump:inexact.gradient}
\item {\bf Inexact Stationarity Metric:} There exists $\kappa_{\textup{stat}}\in\R_{+}$, independent of $k$, such that
\begin{equation}\label{eq:assump.inexact.psi}
  |\psi_k(r_0)-\bar{\psi}_k(r_0)| \le \kappa_{\textup{stat}} [\eta_{\textup{stat}}\min\{\psi_k(r_0),\Delta_k,\xi_k\}]^{\zeta_{\textup{stat}}},
\end{equation}
where the user-specified positive parameters $\zeta_{\textup{stat}}$, $\eta_{\textup{stat}}$ and $\xi_k$ satisfy
\[
\zeta_{\textup{stat}}>1,\quad 0<\eta_{\textup{stat}}<1\quad \text{and}\quad \lim_{k\to+\infty}\xi_k=0.
\]
\label{assump:inexact.Psi}
\end{enumerate}
\end{assumption}

A few comments are in order regarding \Cref{assump:inexact}.
First, \crefpart{assump:inexact}{assump:inexact.object} ensures that the computed reduction $\cred_k$ is a sufficiently accurate approximation of the actual reduction $\ared_k$.
Here, notice that all quantities on the right-hand side of~\eqref{eq:assump.inexact.objective} are available when
computing $\cred_k$.
Second, \crefpart{assump:inexact}{assump:inexact.gradient} requires approximations of the gradient of $f$ within a prescribed tolerance that depends on the state of the algorithm and in particular, the approximate gradient $g_k$.
See \cite[Appendix~B]{Baraldi2023}, which introduces algorithms for satisfying these assumptions when $f(x)$ and $\nabla f(x)$ are approximated by some functions $\hat{f}(x,\tau)$ and $\hat{g}(x,\tau)$, respectively, that produce order $\tau\ge 0$ approximations.
Third, \crefpart{assump:inexact}{assump:inexact.Psi} is motivated by the global convergence analysis in \cite{Baraldi2023}, which we will discuss later in this section.
Notice that $\psi_k(r_0)$, satisfying \eqref{eq:assump.inexact.psi}, must be computed at the same time as the inexact gradient $g_k$, satisfying \eqref{eq:asssump.inexact.gradeint}.
We extend the inexactness computations derived in \cite[Appendix~B]{Baraldi2023} to account for inexact evaluations of $\psi_k(r_0)$ in \Cref{app:inexact}, where we assume that we can compute a Type~1 approximation of the proximity operator.
Finally, we emphasize that $\zeta_{\textup{obj}}$, $\zeta_{\textup{stat}}$, $\eta_{\textup{obj}}$, $\eta_{\textup{stat}}$, $\theta_k$, and $\xi_k$ are introduced primarily for asymptotic convergence guarantees.
In practice, we use, e.g., $\zeta_{\textup{obj}}=\zeta_{\textup{stat}}=(p+1)/p$, $\eta_{\textup{stat}}=(p-1)/p$, $\eta_{\textup{obj}}=\eta_{\textup{stat}}\min\{\eta_1,(1-\eta_2)\}$, and $\theta_k=\xi_k=10^{-\lfloor k/p \rfloor}$ for $p=1000$.

With \Cref{assump:inexact}, we can fully specify the inexact proximal trust-region algorithm, which we list in \Cref{alg:cap}.
\Cref{alg:cap} is seemingly identical to \cite[Algorithm~1]{Baraldi2023}; however, to prove global convergence, one must carefully handle the inexact evaluation of the proximity operator when, e.g., computing trial iterates via the Cauchy point and when evaluating the stationarity metric $\Psi_k$.

\begin{algorithm}
\caption{Proximal Trust-Region Algorithm}\label{alg:cap}
\begin{algorithmic}[1]
\Require $x_1\in \dom \phi$, $\Delta_1>0$, $0<\eta_1<\eta_2<1$, and $0<\gamma_1\le \gamma_2 < 1 \le \gamma_3$
\For{$k=1,2,\ldots$}
    \State \textbf{Model Selection:} Choose $m_k$ with $g_k$ satisfying \crefpart{assump:inexact}{assump:inexact.gradient} \label{alg:assump.model.inexact}
    \State \textbf{Step Computation:} Compute $x_k^+\in X$ satisfying equation~\eqref{eq:FCD.condition} \label{alg:exact.inexact}
    \State \textbf{Computed Reduction:} Compute $\cred_k$ satisfying \crefpart{assump:inexact}{assump:inexact.object}
    \State \textbf{Step Acceptance and Radius Update:} Compute $\rho_k$ as in \eqref{eq:rho}

    \If{$\rho_k<\eta_1$}
        \State $x_{k+1}\gets x_k$
        \State $\Delta_{k+1}\in [\gamma_1\Delta_k,\gamma_2\Delta_k]$

    \Else
        \State $x_{k+1}\gets x_k^+$
        \If{$\rho_k\in [\eta_1,\eta_2)$}
            \State $\Delta_{k+1}\in[\gamma_2\Delta_k,\Delta_k]$
        \Else
            \State $\Delta_{k+1}\in[\Delta_k,\gamma_3\Delta_k]$
        \EndIf
    \EndIf

\EndFor
\end{algorithmic}
\end{algorithm}

To conclude this discussion, we state the global convergence result of \Cref{alg:cap}.
We note that the proof of this result is nearly identical to the proofs of \cite[Theorem~1]{baraldi2024efficient} and \cite[Theorem~3]{Baraldi2023}, but requires some modification to account for the inexact proximity operator evaluations.
\begin{theorem}[Global Convergence]\label{thm:global.convergence}
  Suppose \Cref{as:probdata} and \Cref{assump:inexact} hold and let $\{x_k\}$ be the sequence of iterates generated by \Cref{alg:cap}.
  If
  \[
    \sum_{k=1}^{+\infty}\frac{1}{1+\displaystyle{\max_{j=1,\ldots,k}}\|B_j\|} = +\infty,
  \]
  then
  \begin{equation}\label{eq:liminf.inexact}
    \liminf_{k\to+\infty}\Psi_k=0.
  \end{equation}
  In addition, suppose that there exists $r_{\max}$ for which $r_k\le r_{\max}$ for all $k$ and that $x_k(r_k)\approx_1\Prox_{r_k\phi}(x_k-r_kg_k)$ with $\varepsilon_k$-precision, where $\varepsilon_k\in\R_+$ satisfies
  \begin{equation}\label{eq:global.conv.eps}
    \varepsilon_k \le \bar\kappa r_k\bar{\Psi}_k
  \end{equation}
  for some fixed $\bar\kappa\in(0,1)$.
  Then, we have that
  \begin{equation}\label{thm.conv.Psi.x.r0}
    \liminf_{k\to \infty} \psi_k(r_0) = 0
  \end{equation}
  and
  \begin{equation}\label{thm.conv.Psi.x}
    \liminf_{k\to \infty} \frac{1}{r}\|\Prox_{r\phi}(x_k-r\nabla f(x_k))-x_k\| = 0 \quad\forall\, r\in\R_{++}.
  \end{equation}
\end{theorem}
\begin{proof}
  The proof of \eqref{eq:liminf.inexact} is identical to that of \cite[Theorem~3]{Baraldi2023}.
  We briefly describe the notational changes here.
  We first note that \cite[Lemmas~8~\&~9]{Baraldi2023} hold with $\omega_k=\|B_k\|$ and $h_k=\Psi_k$ since they only depend on the fraction of Cauchy decrease condition \eqref{eq:FCD.condition}, the properties of $f$ and the trust-region radius update mechanism.
  Moreover, \cite[Lemma~6~\&~Corollary~2]{Baraldi2023} hold because of \crefpart{assump:inexact}{assump:inexact.object}.
  These are the necessary results for the proof of \cite[Theorem~3]{Baraldi2023} and hence, \eqref{eq:liminf.inexact} follows by identical arguments.

  To prove \eqref{thm.conv.Psi.x.r0} and \eqref{thm.conv.Psi.x}, \Cref{lem:inexac.exact.prox.diff}, \eqref{eq:global.conv.eps} and the nonexpansivity of the proximity operator combine to ensure that
  \[
    \bar{\Psi}_k \le \Psi_k + \frac{\varepsilon_k}{r_k}
    \le \Psi_k + \bar\kappa \bar{\Psi}_k
  \]
  and therefore, $(1-\bar\kappa)\bar{\Psi}_k\le \Psi_k$ proving that the limit inferior of $\bar{\Psi}_k$ is zero since $\bar\kappa<1$.
  Using this fact, we can prove \eqref{thm.conv.Psi.x.r0}.
  In particular, \Cref{lem:prox.decrease} ensures that
  \[
    \left\{\begin{array}{ll}
      \bar{\psi}_k(r_0) \le \bar{\Psi}_k & \text{if $r_0\ge r_k$} \\
      \bar{\psi}_k(r_0) \le \frac{r_{\max}}{r_0}\bar{\Psi}_k & \text{if $r_0<r_k$.}
    \end{array}\right.
  \]
  Now, by \crefpart{assump:inexact}{assump:inexact.Psi}, if $\kappa_{\textup{stat}}=0$, then $\psi_k(r_0)=\bar\psi_k(r_0)$ and \eqref{thm.conv.Psi.x.r0} follows.
  Otherwise, the fact that $\xi_k\to 0$ guarantees the existence of $K_{\textup{stat}}\in\mathbb{N}$ for which $\xi_k\le\kappa_{\textup{stat}}^{-1/(\zeta_{\textup{stat}}-1)}$ for all $k\ge K_{\textup{stat}}$ and therefore,
  \[
    |\psi_k(r_0)-\bar{\psi}_k(r_0)|
    \le \kappa_{\textup{stat}}[\eta_{\textup{stat}}\min\{\psi_k(r_0),\Delta_k\}]\xi_k^{\zeta_{\textup{stat}}-1}
    \le \eta_{\textup{stat}}\min\{\psi_k(r_0),\Delta_k\}
  \]
  for all $k\ge K_{\textup{stat}}$.
  Consequently, the preceding estimate ensures that
  \[
    \psi_k(r_0) \le \bar{\psi}_k(r_0) + \eta_{\textup{stat}}\psi_k(r_0) \qquad\implies\qquad
    (1-\eta_{\textup{stat}})\psi_k(r_0) \le \bar{\psi}_k(r_0)
  \]
  and \eqref{thm.conv.Psi.x.r0} follows since $\eta_{\textup{stat}}<1$.  Finally, to prove \eqref{thm.conv.Psi.x}, the nonexpansivity of the proximity operator and \crefpart{assump:inexact}{assump:inexact.gradient} yield
  \begin{align*}
    \frac{1}{r_0}\|\Prox_{r_0\phi}(x_k-r_0\nabla f(x_k))-x_k\|
    &\le \|g_k-\nabla f(x_k)\| + \bar\psi_k(r_0) \\
    &\le \kappa_{\textup{grad}}\psi_k(r_0) + \bar{\psi}_k(r_0),
  \end{align*}
  the left-hand side of which has limit inferior equal to zero since $\psi_k(r_0)$ and $\bar{\psi}_k(r_0)$ do.
  The result \eqref{thm.conv.Psi.x} then follows as before using \Cref{lem:prox.decrease}.
\end{proof}

Upon first glance, the condition \eqref{eq:global.conv.eps} appears to be impractical.  However, in the subsequent subsection, we analyze an algorithm (\Cref{alg:inex.epsilon}) that is guaranteed to satisfy \eqref{eq:global.conv.eps} in finitely many iterations.
In addition, we note that if one wishes to use $\Psi_k$ as a stopping condition for \Cref{alg:cap}, \Cref{assump:inexact} needs to be modified so that $\psi_k(r_0)$ is replaced by $\Psi_k$ in \eqref{eq:asssump.inexact.gradeint} and \eqref{eq:assump.inexact.psi}.
However, this introduces significant computational challenge since $r_k$, $x_k(r_k)$, $\Psi_k$ and $g_k$ would all be interdependent.
As such, it is often more convenient to fix $r_0$ and enforce \Cref{assump:inexact} as stated.
\Cref{thm:global.convergence} then ensures that $\psi_k(r_0)$ will fall below a prescribed tolerance within finitely many iterations (cf.\ \eqref{thm.conv.Psi.x.r0}).
Finally, \Cref{thm:global.convergence}, particularly \eqref{thm.conv.Psi.x}, demonstrates that the sequence of iterations $\{x_k\}$ admits a subsequence that is asymptotically stationary for the original problem \eqref{eq:optprob}.

\subsection{Inexact Proximal Cauchy Point}\label{subsec:gcp}
A critical component of the convergence theory for Algorithm~\ref{alg:cap} is that the trial iterates satisfy \eqref{eq:FCD.condition}.
In the nonsmooth case, \cite{Baraldi2023} instead defines the Cauchy point using the exact proximal gradient path, $\bar{p}_k(r)$ where the steplength $r_k$ is chosen to satisfy Armijo-Goldstein-type conditions on the inexact proximal gradient path.
These conditions, which \cite{Baraldi2023} satisfy using a bidirectional proximal search, are problematic when the proximity operator is evaluated inexactly since the proof that $r_k$ exists depends heavily on the monotonicity properties from \Cref{lem:prox.decrease}.
Instead, we employ the simplified Cauchy point introduced in \cite{baraldi2024efficient}, which is defined by
\begin{equation}\label{eq:cp}
  x_k^c = x_k + \alpha_k p_k(r_k),
\end{equation}
where $r_k\in[r_{\min},r_{\max}]$ with $0<r_{\min}\le r_{\max}<+\infty$ and $\alpha_k\in[0,1]$ is computed by solving the one-dimensional quadratic optimization problem
\[
\begin{aligned}
  &\min_{\alpha} \frac{\alpha}{2}\inner{B_k p_k(r_k)}{p_k(r_k)} + \alpha(\inner{g_k}{p_k(r_k)}+\phi(x_k(r_k))-\phi(x_k)) \\
  &\textup{subject to}\quad
  0 \le \alpha \le \min\left\{1,\frac{\Delta_k}{\|p_k(r_k)\|}\right\}.
\end{aligned}
\]
In \cite{baraldi2024efficient}, $r_k$ is chosen to be the safeguarded spectral step length
\[
  r_k=\max\{r_{\min},\min\{r_{\max},r_{k,0}\}\}
  \quad\text{for}\quad
  r_{k,0}=\left\{\begin{array}{ll}
  \frac{\|g_k\|^2}{\inner{B_k g_k}{g_k}} & \text{if $\inner{B_k g_k}{g_k} > 0$} \\
  \frac{r_0}{\|g_k\|} & \text{otherwise.}
  \end{array}\right.
\]
However, other selections are possible.
For example, one could compute $r_k$ by applying the bidirectional search algorithm introduced in \cite[Algorithm~2]{Baraldi2023}, possibly exiting prematurely so that $r_k\in[r_{\min},r_{\max}]$.
Notice that the bounds on $\alpha_k$ ensure that $x_k^c$ satisfies \eqref{eq:FCD.condition.1} with $\kappa_{\textup{rad}}=1$. 
In order to prove that $x_k^c$ also satisfies \eqref{eq:FCD.condition.2} --- and hence an acceptable trial iterate $x_k^+$ exists --- we must use the descent properties of \Cref{lem:decr.technical}.
In particular, according to \crefpart{lem:decr.technical}{lem:decr.technical.1}, for each $r\in\R_{++}$, there exists $\varepsilon=\varepsilon(r)\in\R_{++}$ for which
\[
  Q_k(r)\coloneqq \inner{g_k}{p_k(r)} + \phi(x_k(r))-\phi(x_k)
\]
computed using $x_k(r)\approx_1\Proxrphi(x_k-rg_k)$ with $\varepsilon$-precision, satisfies
\begin{equation}\label{eq:decr.inexact}
  Q_k(r) \le -\frac{\kappa_{\rm desc}}{r}\|p_k(r)\|^2,
\end{equation}
where $\kappa_{\textup{desc}}\in(0,1)$.

The relative accuracy condition in \crefpart{lem:decr.technical}{lem:decr.technical.2} motivates a procedure for controlling $\varepsilon$ so that $x_k(r)$ satisfies the inexact descent condition \eqref{eq:decr.inexact} --- and ultimately the requirement \eqref{eq:global.conv.eps} for global convergence in \Cref{thm:global.convergence} --- 
in finite time.
\begin{algorithm}[htb!]
\caption{Inexact Proximal Gradient Evaluation}
\begin{algorithmic}[1]
\Require{$x_k,\,g_k\in X$, $r,\,\varepsilon^{(0)}\in\R_{++}$, $\beta\in(0,1)$, $\kappa\in(0,\tfrac12)$, and $\ell=0$}
\State $u^{(0)}\approx_1 \Proxrphi(x_k-rg_k)$ with $\varepsilon^{(0)}$-precision
\While{$\varepsilon^{(\ell)} > \kappa\|u^{(\ell)}-x_k\|$}
  \State{$\ell \gets \ell+1$}
  \State{$\varepsilon^{(\ell)}\gets \beta\varepsilon^{(\ell-1)}$}
  \State $u^{(\ell)}\approx_1\Proxrphi(x_k-rg_k)$ with $\varepsilon^{(\ell)}$-precision
\EndWhile
\end{algorithmic}
\label{alg:inex.epsilon}
\end{algorithm}
\Cref{alg:inex.epsilon} describes such a method for computing $\varepsilon$ by systematically decreasing $\varepsilon$, i.e., $\varepsilon\gets\beta\varepsilon$ for fixed $\beta\in(0,1)$.
It is straightforward to show from \Cref{lem:decr.technical} 
that \Cref{alg:inex.epsilon} is guaranteed to terminate in finitely many iterations, which we do in the next result.
\begin{proposition}
Fix $r\in\R_{++}$. If $x_k\neq \bar{x}_k(r)$, then \Cref{alg:inex.epsilon} terminates in at most
\[
  \bar\ell = \max\left\{0,\left\lceil \log_\beta\left(\kappa\|\bar{x}_k(r)-x_k\|\right)-\log_\beta((1+\kappa)\varepsilon^{(0)})\right\rceil\right\}
\]
iterations with $u^{(\ell)}=x_k(r)$ satisfying
\eqref{eq:decr.inexact}. 
\end{proposition}
\begin{proof}
  The bound on the total number of iterations is then produced by noting that \Cref{lem:inexac.exact.prox.diff} ensures
  \[
    \|u^{(\ell)}-\bar{x}_k(r)\|\le \varepsilon^{(\ell)}
  \]
  and so the reverse triangle inequality yields
  \[
    \|u^{(\ell)}-x_k\| \ge \|\bar{x}_k(r)-x_k\|-\varepsilon^{(\ell)}.
  \]
  Therefore, \Cref{alg:inex.epsilon} will terminate when
  \[
    \varepsilon^{(\ell)} \le \kappa(\|\bar{x}_k(r)-x_k\|-\varepsilon^{(\ell)})
    \qquad\iff\qquad
    \varepsilon^{(\ell)} \le \frac{\kappa}{1+\kappa}\|\bar{x}_k(r)-x_k\|.
  \]
  Using the fact that $\varepsilon^{(\ell)}=\varepsilon^{(0)}\beta^\ell$ we see that any
  \[
    \ell \ge \log_\beta(\kappa\|\bar{x}_k(r)-x_k\|)-\log_\beta((1+\kappa)\varepsilon^{(0)})
  \]
  satisfies the stopping conditions in \Cref{alg:inex.epsilon}, proving the iteration bound.
  The fact that the computed $x_k(r)$ satisfies \eqref{eq:decr.inexact} follows from \crefpart{lem:decr.technical}{lem:decr.technical.1} since there exists $\ell$ for which $\varepsilon^{(\ell)}\le\kappa\|u^{(\ell)}-x_k\|$.
  Plugging this into \crefpart{lem:decr.technical}{lem:decr.technical.1} yields \eqref{eq:decr.inexact} with $\kappa_{\textup{desc}}=\frac{1-\kappa^2}{2}$.
\end{proof}

Note that if we use \Cref{alg:inex.epsilon} to compute $x_k(r_k)$, then $\Psi_k$ satisfies \eqref{eq:global.conv.eps} with constant $\bar\kappa=\kappa/(1-\kappa)\in(0,1)$, provided that $\kappa\in(0,\tfrac12)$.
We are now in position to prove the existence of such a trial iterate $x_k^+$ satisfying \eqref{eq:FCD.condition}.
To prove this result, it suffices to show that the simplified Cauchy point, $x_k^c$, satisfies \eqref{eq:FCD.condition}.
\begin{theorem}
\label{thm:fcd}
  Suppose $\Psi_k>0$, where $r_k\in[r_{\min},r_{\max}]$, $x_k(r_k)\approx_1\Prox_{r_k\phi}(x_k-r_kg_k)$ with $\varepsilon_k$-precision, and $\varepsilon_k$ is chosen sufficiently small so that \eqref{eq:decr.inexact} holds.
  Then, there exists $x_k^+\in X$ that satisfies \eqref{eq:FCD.condition} with
  \[
  \kappa_{\textup{fcd}}=\tfrac{1}{2}\kappa_{\textup{desc}}\min\{\kappa_{\textup{desc}},r_{\min}\}
  \qquad\text{and}\qquad \kappa_{\textup{rad}}=1.
  \]
\end{theorem}
\begin{proof}
  The result follows if $x_k^c$ satisfies \eqref{eq:FCD.condition}, and then setting $x_k^+=x_k^c$.
  To prove that $x_k^c$ satisfies \eqref{eq:FCD.condition}, we first notice that \eqref{eq:FCD.condition.1} holds with $\krad=1$ because of the upper bound on $\alpha_k$.
  Now to prove that \eqref{eq:FCD.condition.2} holds, we follow the proof of \cite[Proposition~2]{baraldi2024efficient}.
  In language of the proof of \cite[Proposition~2]{baraldi2024efficient}, we define the following quantities
  \[
  \begin{aligned}
    s_k&\coloneqq p_k(r_k), \qquad
    &&\alpha_{k,\max}\coloneqq \min\{1,\Delta_k/\|s_k\|\},\\ 
    \kappa_k &\coloneqq \inner{B_kp_k(r_k)}{p_k(r_k)}, \qquad 
    &&d_k      \coloneqq \inner{g_k}{p_k(r_k)} + \phi(x_k(r_k)) - \phi(x_k), \\
    q_k(\alpha) &\coloneqq \frac{\kappa_k}{2}\alpha^2 + d_k\alpha.&&
  \end{aligned}
  \]
  Here, we assume that $\varepsilon_k=\varepsilon_k(r_k)$ is chosen sufficiently small so that \eqref{eq:decr.inexact} holds (e.g., via \Cref{alg:inex.epsilon}).
  Recall that $d_k \le -\kappa_{\textup{desc}}r_k \Psi_k^2$ from \eqref{eq:decr.inexact}, $r_k\Psi_k=\|s_k\|$ and the unconstrained minimizer of $q_k$ is
  $-d_k/\kappa_k$ when $\kappa_k>0$.
  When $\kappa_k>0$, we have $\alpha_k = \min\{-d_k/\kappa_k,\alpha_{k,\max}\}$. When
  $\kappa_k=0$, $q_k(\alpha) = d_k\alpha \le -\kappa_{\textup{desc}}r_k \Psi_k^2\alpha$ and $\alpha_k = \alpha_{k,\max}$.
  Finally, if $\kappa_k < 0$, then
  $q_k$ is concave and $\alpha_k$ is either 0 or $\alpha_{k,\max}$.
  In this case, we see that
  \[
    q_k(\alpha_{k,\max}) \le -\kappa_{\textup{desc}}\Psi_k \min\{r_k \Psi_k,\Delta_k\} < 0 = q_k(0)
  \]
  and hence $\alpha_k=\alpha_{k,\max}$.
  Consequently, there are three
  cases to handle: $\alpha_k=1$, $\alpha_k=\Delta_k/\norm{s_k}$, and
  $\alpha_k=-d_k/\kappa_k$.
  For these cases, we rely heavily on the bound
  \[
    m_k(x_k)-m_k(x_k+\alpha s_k) \ge -\frac{\kappa_k}{2}\alpha^2 - d_k\alpha \ge -q_k(\alpha) \quad\forall\,\alpha\in [0,1].
  \]
  Treatment of the three cases is nearly identical to the proof of \cite[Proposition~2]{baraldi2024efficient} with the addition of the constant $\kappa_{\textup{desc}}$.
  We include the proof here for completeness.

  \noindent
  {\bf Case $\alpha_k=1$:} If $\kappa_k \le 0$, then
  the facts $1+\|B_k\| \ge 1$ and $r_k \ge r_{\min}$ ensure
  \[
    m_k(x_k) - m_k(x_k^c) \ge -d_k
                          \ge \kappa_{\textup{desc}}r_k \Psi_k^2
                          \ge\kappa_{\textup{desc}} r_{\min} \frac{\Psi_k^2}{1+\|B_k\|}.
  \]
  If $\kappa_k > 0$, then the unconstrained minimizer of $q_k$ satisfies
  $-d_k/\kappa_k\ge 1$, which is equivalent to
  $\kappa_k \le -d_k$. Consequently,
  \[
    m_k(x_k) - m_k(x_k^c)
      \ge -\frac{1}{2}\kappa_k-d_k
      \ge -\frac{1}{2} d_k
      \ge \kappa_{\textup{desc}}\frac{r_{\min}}{2}\frac{\Psi_k^2}{1+\|B_k\|}.
  \]

  \noindent
  {\bf Case $\alpha_k=\Delta_k/\norm{s_k}$:} If $\kappa_k\le 0$, then
  $\alpha_k \le 1$ and
  \begin{align*}
    m_k(x_k) - m_k(x_k^c) \ge -\alpha_k d_k
                          \ge \kappa_{\textup{desc}}\frac{\Delta_k}{\norm{s_k}}r_k\Psi_k^2
                           =  \kappa_{\textup{desc}}\Delta_k \Psi_k.
  \end{align*}
  If $\kappa_k > 0$, then
  $\alpha_k=\Delta_k/\norm{s_k}\le-d_k/\kappa_k$. Consequently,
  \[
    m_k(x_k) - m_k(x_k+\alpha_k s_k) = \alpha_k(-\tfrac12\alpha_k\kappa_k - d_k)
      \ge -\frac{\alpha_k}{2}d_k
      \ge \frac{\kappa_{\textup{desc}}}{2}\Delta_k\Psi_k.
  \]

  \noindent
  {\bf Case $\alpha_k=-d_k/\kappa_k$:} In this case,
  $0<-d_k\le\kappa_k\le\norm{B_k}\norm{s_k}^2$ and
  \[
    m_k(x_k)-m_k(x_k^c) \ge \frac{1}{2} \frac{d_k^2}{\kappa_k}
    \ge \frac{\kappa_{\textup{desc}}^2}{2}\frac{r_k^2\Psi_k^4}{\norm{B_k}\norm{s_k}^2}
    \ge \frac{\kappa_{\textup{desc}}^2}{2}\frac{\Psi_k^2}{1+\norm{B_k}}.
  \]
  Combining cases 1, 2 and 3 proves that \eqref{eq:FCD.condition.2} holds for $x_k^c$ as desired.
\end{proof}

\subsection{Subproblem Solvers}
\label{sec:sub}
The Cauchy point analyzed in the previous subsection plays a pivotal role in the computation of the trial iterates $x_k^+$.
Building on the Cauchy point, we now provide a brief description of the necessary modifications to the subproblem solvers presented
in \cite[Algorithms~3~\&~4]{baraldi2024efficient}, which are referred to in that source as {\tt SPG2} and {\tt NCG}.
These methods leverage the exact proximity operator to generate trial iterates $x_k^+$ that satisfy the fraction of Cauchy decrease condition \eqref{eq:FCD.condition}
and improve upon the Cauchy point using a descent method, thereby ensuring global convergence of \Cref{alg:cap}.
These methods are iterative and compute steps at the $j$-th inner iteration based on the proximal gradient
\begin{equation}\label{eq:spgstep}
  s_{k,j} = \Prox_{t_{k,j}\phi}(x_{k,j} - t_{k,j}\nabla f_k(x_{k,j})) - x_{k,j},
\end{equation}
where
$t_{k,j}\in[t_{\min},t_{\max}]$ is the safeguarded spectral step
length given by $t_{k,j}=\max\{t_{\min},\min\{t_{\max},\bar{t}_{k,j}\}\}$ with 
\[
    \bar{t}_{k,j} = \left\{\begin{array}{ll}
        \frac{\|s_{k,j-1}\|^2}{\inner{s_{k,j-1}}{B_k s_{k,j-1}}} & \text{if $\inner{s_{k,j-1}}{B_ks_{k,j-1}} >0$} \\
        \frac{r_k}{\norm{d_{k,j}}} & \text{otherwise,}
    \end{array}\right.
\]
where $d_{k,j}$ is the subproblem gradient. 
Clearly, when the proximity operator is computed inexactly, \eqref{eq:spgstep} is modified to
\[
  s_{k,j}=u_{k,j}-x_{k,j},
\]
where $u_{k,j}\approx_1\Prox_{t_{k,j}\phi}(x_{k,j}-t_{k,j}\nabla f_k(x_{k,j}))$ with $\varepsilon_j$-precision.  We compute $\varepsilon_j$ using \Cref{alg:inex.epsilon} with $\varepsilon^{(0)}=\varepsilon_{j-1}$.
In practice, we terminate the inexact proximity operator evaluation when either the criterion on the error $\varepsilon_j$ is satisfied
or when the decrease condition in \eqref{eq:decr.inexact} is satisfied. 
Consequently, the steps computed by {\tt SPG2} and {\tt NCG} decrease the subproblem model (cf.~\Cref{lem:decr.technical}), ensuring that \eqref{eq:FCD.condition} holds.
Moreover, we terminate the subproblem routine if a fixed iteration limit is exceeded or when the stopping criterion
\[
  \frac{1}{t_{k,j}}\|s_{k,j}\| \le \min\{\tau_{\textup{abs}},\tau_{\textup{rel}} \Psi_k^{1+\alpha}\}
\]
is satisfied.
Here, $\tau_{\textup{abs}},\, \tau_{\textup{rel}}>0$ are absolute and relative tolerances, respectively, and $\alpha\ge 0$ controls the rate of convergence of \Cref{alg:cap}, cf.\ \cite[Theorem~3]{baraldi2024local}. 
We mention here that the spectral projected gradient method ({\tt SPG}) introduced in \cite[Algorithm~5]{Baraldi2023} and the dogleg method ({\tt Dogleg}) introduced in \cite[Algorithm~2]{baraldi2024efficient} require less straightforward modifications to account for inexact evaluations of the proximity operator.
For example, {\tt SPG} requires computing a zero of the nonlinear equation
\[
  t\mapsto \|\Prox_{rt\phi}(x_k+t(x-x_k))-x_k\|-\Delta_k
\]
to evaluate the proximity operator of $x$ associated with $\phi$, accounting for the trust-region constraint, while certain instances of {\tt Dogleg} require the application of the generalized Jacobian of the proximity operator. 
As such, we do not study these methods further in this work.

\section{Iterative Methods for Computing Proximity Operators}
\label{sec:compprox}
Methods for approximately computing proximity operators are typically tailored to the specific structure of the nonsmooth function $\phi$.  In this section, we describe two such procedures.  The first applies to structured nonsmooth functions with the form
\begin{equation}\label{eq:struc.nonsmooth}
  \phi(x) = \phi_0(x)+\phi_1(Dx-b),
\end{equation}
where $Y$ is another Hilbert space, $\phi_0:X\to(-\infty,+\infty]$ and $\phi_1:Y\to(-\infty,+\infty]$ are proper, closed and convex, $D:X\to Y$ is a continuous linear map and $b\in Y$.
The second applies to the case of computing a proximity operator using weighted inner products.

\subsection{Structured Nonsmooth Optimization}
Since $\phi$ in \eqref{eq:struc.nonsmooth} is the sum of nonsmooth functions and involves the composition with an affine map, there is generally no closed-form expression for its proximity operator.
Instead, we can approximate the proximity operator using an iterative method applied to the dual problem
\[
  \max_{y\in Y} \left\{ \min_{x\in X} \left\{\frac{1}{2r}\|x-z\|^2+\phi_0(x)+(y,Dx-b)\right\} - \phi_1^*(y)\right\}.
\]
Defining $d:Y\to\R$ by
\[
  d(y)\coloneqq\min_{x\in X} \left\{\frac{1}{2r}\|x-z\|^2+\phi_0(x)+\inner{y}{Dx-b}\right\},
\]
we see that $d$ is concave and Fr\'{e}chet differentiable with Lipschitz continuous gradient and the dual problem can be equivalently written as the maximization problem
\begin{equation}\label{eq:struct.prox.dual}
  \max_{y\in Y}\; \{v(y)\coloneqq d(y)-\phi_1^*(y)\}.
\end{equation}
We can solve the dual problem \eqref{eq:struct.prox.dual} using the spectral proximal gradient method as in \cite{kouri2026tr}, which we include as \Cref{alg:struct.prim.dual}.
\Cref{alg:struct.prim.dual} produces the primal-dual iterates $(x^{(\ell)},y^{(\ell)})$ given by
\begin{equation}\label{eq:struct.prim.dual}
\left\{\begin{aligned}
  x^{(\ell)}&=\Prox_{r\phi_0}(z-rD^*y^{(\ell)}) \\
  y^{(\ell+1)}&=y^{(\ell)}+\lambda^{(\ell)}[\Prox_{\gamma^{(\ell)}\phi_1^*}(y^{(\ell)}+\gamma^{(\ell)}(Dx^{(\ell)}-b))-y^{(\ell)}],
\end{aligned}\right.
\end{equation}
where $0<\gamma_{\min}\le\gamma_{\max}$ are safeguard parameters, $\gamma^{(\ell)}\in[\gamma_{\min},\gamma_{\max}]$ is the safeguarded Barzilai--Borwein step length, and $\lambda^{(\ell)}$ is chosen according to a monotone or nonmonotone line search.
\begin{algorithm}[!ht]
\caption{Primal-Dual Proximity Operator Computation}
  \label{alg:struct.prim.dual}
  \begin{algorithmic}[1]
    \Require{Proximity operator arguments $z\in X$ and $r\in\R_{++}$, initial guess
      $y^{(1)}\in\dom{\,\phi_1^*}$ and $V^{(1)}=v(y^{(1)})$, safeguards $0<\gamma_{\min}<\gamma_{\max}<+\infty$ and $a_{\min}\in(0,1]$,
      initial steplengths $\gamma^{(1)}\in[\gamma_{\min},\gamma_{\max}]$ and
      $\lambda_0\in(0,1]$, and $0<\sigma_1<\sigma_2<1$, and positive parameters
      $\alpha\in(0,1)$, $\beta\in[\sigma_1,\sigma_2]$ and $\varepsilon\in\R_{++}$}
    \For{$\ell=1,2,\ldots$}
      \State{$x^{(\ell)}\gets\Prox_{r\phi_0}(z-rD^*y^{(\ell)})$}
      \If{$\phi_1(Dx^{(\ell)}-b)+\phi_1^*(y^{(\ell)})-\inp{y^{(\ell)}}{Dx^{(\ell)}-b} \le \varepsilon^2/(2r)$}
        \State{\bf break}
      \EndIf
      \State{$\mu^{(\ell)}\gets Dx^{(\ell)}-b$}
      \If{$\ell>1$}
        \State{$\gamma^{(\ell)}\gets\min\left\{\gamma_{\max},\max\left\{\gamma_{\min},\displaystyle{\frac{\lambda^{(\ell-1)}\|s^{(\ell-1)}\|_Y^2}{(\mu^{(\ell-1)}-\mu^{(\ell)},s^{(\ell-1)})_Y}}\right\}\right\}$}
      \EndIf
      \State{$s^{(\ell)}\gets\Prox_{\gamma^{(\ell)}\phi_1^*}(y^{(\ell)}+\gamma^{(\ell)}\mu^{(\ell)})-y^{(\ell)}$}
      \State{$\lambda^{(\ell)}\gets\lambda_0$}
      \State{$\mathcal{Q}^{(\ell)}\gets(\mu^{(\ell)},s^{(\ell)})_Y+\phi_1^*(y^{(\ell)})-\phi_1^*(y^{(\ell)}+s^{(\ell)})$}
      \While{$v(y^{(\ell)}+\lambda^{(\ell)} s^{(\ell)}) < V^{(\ell)} + \alpha\lambda^{(\ell)}\mathcal{Q}^{(\ell)}$}
        \State{$\delta\gets-\displaystyle{\frac{[\lambda^{(\ell)}(\mu^{(\ell)},s^{(\ell)})_Y+\phi_1^*(y^{(\ell)})-\phi_1^*(y^{(\ell)}+\lambda^{(\ell)}s^{(\ell)})]}{2[d(y^{(\ell)}+\lambda^{(\ell)} s^{(\ell)})-d(y^{(\ell)})-\lambda^{(\ell)}(\mu^{(\ell)},s^{(\ell)})_Y]}}$}
        \If{$\delta\in[\sigma_1,\sigma_2]$}
          \State{$\lambda^{(\ell)}\gets\delta\lambda^{(\ell)}$}
        \Else
          \State{$\lambda^{(\ell)}\gets\beta\lambda^{(\ell)}$}
        \EndIf
      \EndWhile
      \State{$y^{(\ell+1)}\gets y^{(\ell)} + \lambda^{(\ell)}s^{(\ell)}$}
      \State{$V^{(\ell+1)}\gets(1-a^{(\ell)})V^{(\ell)}+a^{(\ell)}v(y^{(\ell+1)})$ for $a^{(\ell)}\in[a_{\min},1]$}
    \EndFor
  \end{algorithmic}
\end{algorithm}
Although \cite{kouri2026tr} is focused on the case when $\phi_1$ is Lipschitz continuous, \cite[Proposition~6]{kouri2026tr} applies to our case of extended real-valued $\phi_1$, ensuring that the sequence of dual variables (ignoring the stopping conditions on line~3) converges weakly to some $\bar{y}$ and, under additional assumptions, the primal sequence converges strongly to the proximity operator $\Proxrphi(z)=\Prox_{r\phi_0}(z-rD^*\bar{y})$.

The next result demonstrates that \Cref{alg:struct.prim.dual}, when terminated prematurely according to the duality-gap stopping condition
in line~3 produces a Type-2 (and hence Type-1) approximation of $\Proxrphi(z)$ with $\varepsilon$-precision.
\begin{proposition}\label{prop:primal.dual.inexact}
Fix $z\in X$ and $r,\,\varepsilon\in\R_{++}$, and let $\{(x^{(\ell)},y^{(\ell)})\}_\ell$ be a sequence generated by \Cref{alg:struct.prim.dual}.
Moreover, suppose $\bar{\ell}\in\mathbb{N}$ is the first iteration for which
\begin{equation}\label{eq:struct.prim.dual.gap}
\phi_1(Dx^{(\ell)}-b)+\phi_1^*(y^{(\ell)})-\inp{y^{(\ell)}}{Dx^{(\ell)}-b} \le \varepsilon^2/(2r)
\end{equation}
holds.
Then, $x^{(\bar{\ell})}\approx_2 \Proxrphi(z)$ with $\varepsilon$-precision.
\end{proposition}
\begin{proof}
    Let $\tau=\varepsilon^2/(2r)$ and denote $\hat{\phi}_1(x)=\phi_1(Dx-b)$.
    The primal-dual iterate $(x^{(\bar{\ell})},y^{(\bar{\ell})})$ satisfies the inequality
    \begin{align*}
    \tau - \phi_1(Dx^{(\bar\ell)}-b)+\inp{y^{(\bar\ell)}}{Dx^{(\bar\ell)}-b} &\geq \phi_1^*(y^{(\bar\ell)})= \sup_{w\in Y}\{\inp{y^{(\bar\ell)}}{w} - \phi_1(w)\},
    \end{align*}
    which implies
    \begin{align*}
    \phi_1(Dx-b) &\geq  \phi_1(Dx^{(\bar\ell)}-b) + \inp{y^{(\bar\ell)}}{D(x-x^{(\bar\ell)})} - \tau \\
    & = \phi_1(Dx^{(\bar\ell)}-b) + \inp{D^*y^{(\bar\ell)}}{x-x^{(\bar\ell)}} - \tau \quad\forall\,x\in X
    \end{align*}
and therefore, $D^*y^{(\bar\ell)}\in \partial_\tau\hat{\phi}_1(x^{(\bar\ell)})$.  Moreover, the primal iterate in \eqref{eq:struct.prim.dual} satisfies
\[
  \frac{1}{r}[(z-rD^*y^{(\bar\ell)})-x^{(\bar\ell)}]\in\partial\phi_0(x^{(\bar\ell)})
\]
by definition of the proximity operator.  Consequently,
\[
\begin{aligned}
  \frac{1}{r}(z-x^{(\bar\ell)}) &= D^*y^{(\bar\ell)} + \frac{1}{r}[(z-rD^*y^{(\bar\ell)})-x^{(\bar\ell)}] \\
  &\in \partial_\tau\hat{\phi}_1(x^{(\bar\ell)})+\partial\phi_0(x^{(\bar\ell)})
  \subseteq\partial_\tau\phi(x^{(\bar\ell)})
  \end{aligned}
\]
and hence \eqref{eq:type2-error} holds.
\end{proof}
When $\phi_1$ is the support function of a nonempty, closed and convex set $C\subseteq Y$, the procedure in \eqref{eq:struct.prim.dual} simplifies further. 
In this setting, $\phi_1^*$ is the indicator function of $C$
and \eqref{eq:struct.prim.dual.gap} becomes
\begin{equation}\label{eq:support.fnt.stop}
-\frac{\varepsilon^2}{2r} \le \inp{y^{(\ell)}}{Dx^{(\ell)}-b} - \phi_1(Dx^{(\ell)}-b) \le 0.
\end{equation}
In addition, \eqref{eq:struct.prim.dual} ensures that each iterate is primal feasible with respect to $\textup{dom}\,\phi_0$ as well as dual feasible, i.e., $y^{(\ell)}\in C$.
We utilize these features in the numerical results to compute the proximity operator of
the sum of the total-variation semi-norm and the indicator function of a convex set.

\subsection{Weighted Inner Products}
\label{ss:weighted.prod}
We now discuss a method for computing the proximity operator using alternative inner products.
This application arises, e.g., when solving discretized infinite-dimensional problems.
In these problems, the discretized (i.e., finite-dimensional) inner product involves a symmetric positive-definite matrix.
Typically, this matrix complicates the proximity operator computation, even if the proximity operator has an analytical form when defined by an alternative inner product such as the Euclidean inner product.
The results in this section first appeared in the technical reports \cite{dprox,maia2026inexactweightedproximaltrustregion}.

Let  $a:X\times X\rightarrow \mathbb{R}$ be a symmetric, coercive, and continuous bilinear form, i.e., there exist constants
$0<\alpha_1\le\alpha_2<\infty$ such that
\[
a(x,y)=a(y,x), \qquad
\alpha_1\|x\|^2 \le a(x,x), \qquad
|a(x,y)| \le \alpha_2\|x\|\|y\|
\qquad \forall\,x,y\in X.
\]
The Lax-Milgram Lemma ensures that $a$ is associated with the invertible, 
self-adjoint, positive, and continuous linear operator $A:X\to X$ given by 
\[
\langle Ax,y\rangle = a(x,y) \quad \forall\,x,\,y\in X.
\]
Here $a(\cdot,\cdot)$ defines an inner product on $X$ with norm $\norm{\cdot}_a = \sqrt{a(\cdot,\cdot)}$. 
We present a method for computing Type-1 approximations to $\Proxrphi(z)$ with $\varepsilon$-precision using the alternative  proximity operator defined by the $A$-weighted inner product, i.e.,
\[
  \Proxrphi^a(x) \coloneqq \operatorname*{arg\,min}_{y\in X} \left\{\phi(y) + \frac{1}{2r}\|y - x\|_a^2\right\},
\]
which we assume to have an analytical form.
\Cref{alg:inex.M.prox} computes a Type-3 approximation (and hence a Type-1 approximation) using 
the $A$-weighted proximal gradient algorithm.  Ignoring the stopping conditions, \cite[Corollary~28.9]{bauschke2017monotone} ensures that the iterates $\{x^{(\ell)}\}$
converge strongly to $\Proxrphi(x)$.
\begin{algorithm}[htb!]
\caption{Weighted Proximal Gradient}
\begin{algorithmic}[1]
\Require{$z\in X$ and $r,\,\delta\in\R_{++}$}
\State $x^{(0)}\gets \Prox^a_{r\phi}(z)$
\For{$\ell=1,2,\ldots$}
  \State{$x^{(\ell)}\gets \Prox_{r\phi}^a(x^{(\ell-1)}-A^{-1}(x^{(\ell-1)}-z))$}
\If{$\|x^{(\ell)}-x^{(\ell-1)}\|_a\le\delta$}
  \State{{\bf break}}
\EndIf
\EndFor
\end{algorithmic}
\label{alg:inex.M.prox}
\end{algorithm}
Let $\bar{\ell}\in\mathbb{N}$ denote the first iteration for which $\|x^{(\bar\ell)}-x^{(\bar\ell-1)}\|_a\le\delta$.
The next result demonstrates that $x^{(\bar\ell)}$ satisfies a similar error bound to that in \Cref{lem:inexac.exact.prox.diff}.

\begin{theorem}\label{app:lemma.tech}
\Cref{alg:inex.M.prox} converges in finitely many iterations.  
Moreover, if $\alpha_1 < \sqrt{2}$ and \Cref{alg:inex.M.prox} exits at iteration $\bar\ell\in\mathbb{N}$, i.e., 
\begin{equation}\label{eq:lemma.tech.1}
  \|x^{(\bar\ell)} - x^{(\bar\ell-1)}\|_a \le \delta,
\end{equation}
then the following error bounds hold:
\[
  \|x^{(\bar\ell-1)}-\Proxrphi(z)\|\le \alpha_1^{-1/2}\left(1-\tfrac{1}{2}\alpha_1^{2}\right)^{-1}\delta
\]
and
\[
  \|x^{(\bar\ell)}-\Proxrphi(z)\|\le \alpha_1^{-1/2}((1-\tfrac{1}{2}\alpha_1^2)^{-1}+1)\delta.
\]
\end{theorem}
\begin{proof}
First, suppose that \Cref{alg:inex.M.prox} produces infinitely many iterations (i.e., \eqref{eq:lemma.tech.1} is never satisfied).
Then, \cite[Corollary~28.9]{bauschke2017monotone} ensures that $x^{(\ell)}$ converges strongly to $\bar{x}=\Proxrphi(z)$, resulting in a contradiction.
Therefore, there exists $\bar\ell$ for which \eqref{eq:lemma.tech.1} holds for any given $\delta>0$.
Denote the $A$-weighted proximal gradient operator by
\[
  G_a(x,t) = \tfrac{1}{t}\left(x-\Prox_{t\phi}^a\left(x-\tfrac{t}{r}A^{-1}(x-z)\right)\right)
\]
and note that $x\mapsto G_a(x,t)$ is strongly monotone for all $t\in(0,2r/\alpha_1^2)$
\cite[Lemma~2]{baraldi2024local}.
Consequently, $x\mapsto G_a(x,r)$ is strongly monotone since $\alpha_1 < \sqrt{2}$ and
\[
  a(G_a(x^{(\ell)},r)-G_a(\xb,r),x^{(\ell)}-\xb) \ge \tfrac{1}{r}\left(1-\tfrac{1}{2}\alpha_1^{2}\right)\|x^{(\ell)}-\xb\|_a^2.
\]
The optimality of $\xb$ for $\phi_x^r$ ensures that $G_a(\xb,r)=0$ and so
\[
  \tfrac{1}{r}\left(1-\tfrac{1}{2}\alpha_1^{2}\right)\|x^{(\ell)}-\xb\|_a \le \|G_a(x^{(\ell)},r)\|_a = \tfrac{1}{r}\|x^{(\ell+1)}-x^{(\ell)}\|_a.
\]
The error bound for $x^{(\bar\ell-1)}$ follows from the above bound with $\ell=\bar\ell-1$, \eqref{eq:lemma.tech.1} and the equivalence of $\|\cdot\|_a$ and $\|\cdot\|$, while the error bound for $x^{(\ell)}$ follows by the triangle inequality.
\end{proof}

The final result in this section
demonstrates that \Cref{alg:inex.M.prox} produces $x^{(\bar\ell)}$ satisfying
$0\in\hat{\partial}_\delta\phi_x^{(r)}(x^{(\bar\ell)})$ and therefore, is a type-3 approximation according to \Cref{lem:delta-prox-char}.  Before proving this, we require the following technical lemma.
\begin{lemma}\label{lem:M.weight.1}
Consider the sequence $\{x^{(\ell)}\}$ generated by \Cref{alg:inex.M.prox}. Then,
\begin{equation}\label{eq:M.weight.1}
\frac{1}{r}(A-I)(x^{(\ell-1)}-x^{(\ell)}) \in \partial\phi^{(r)}_z(x^{(\ell)}),
\end{equation}
which, by definition, is equivalent to
\begin{align}
\phi^{(r)}_z(x) 
&\ge \phi^{(r)}_z(x^{(\ell)}) + \frac{1}{r}\inner{(A-I)(x^{(\ell-1)}-x^{(\ell)})}{x-x^{(\ell)}}
\quad\forall\,x\in X.
\label{eq:diff}
\end{align}
\end{lemma}
\begin{proof}
Denote the subdifferential of $\phi$ with respect to the $A$-weighted inner product by $\partial_a\phi$ and note that $\partial_a\phi=A^{-1}\partial\phi$.
Optimality of the proximity operator implies
\begin{align*}
\frac{1}{r}(x^{(\ell-1)}-x^{(\ell)}-A^{-1}(x^{(\ell-1)}-z)) \in \partial_a \phi(x^{(\ell)}) 
\end{align*}
and adding $(1/r)A^{-1}(x^{(\ell)}-z)$ to both sides of this inclusion yields
\begin{align*}
\frac{1}{r}(x^{(\ell-1)}-x^{(\ell)}-A^{-1}(x^{(\ell-1)}-x^{(\ell)})) \in 
A^{-1}\partial\phi^{(r)}_z(x^{(\ell)}).
\end{align*}
This is equivalent to \eqref{eq:M.weight.1}, concluding the proof. 
\end{proof}

Using \Cref{lem:M.weight.1}, we can now prove that if $x^{(\bar\ell)}$ satisfies the stopping conditions of \Cref{alg:inex.M.prox}, then it is a Type-3 approximation of $\Proxrphi(z)$ with $\varepsilon$-precision.
\begin{theorem}\label{thm:inexact.M.prox}
Let $\bar\ell\in\mathbb{N}$ be the smallest iteration index for which $x^{(\bar\ell)}$ satisfies the $\norm{x^{(\bar\ell)} - x^{(\bar\ell-1)}}_a\le \delta$. 
If
\(
  \delta \le \varepsilon\sqrt{\alpha_1} / (1+\alpha_2),
\)
then $x^{(\bar\ell)}$ is a Type-3 approximation of $\Proxrphi(z)$ with $\varepsilon$-precision.
\end{theorem}
\begin{proof}
Note that the existence of $\alpha_1$ and $\alpha_2$ ensures that
\[
\begin{aligned}
  \frac{1}{r}\|(A-I)(x^{(\bar\ell-1)}-x^{(\bar\ell)})\| &\le \frac{1+\alpha_2}{r} \|x^{(\bar\ell-1)}-x^{(\bar\ell)}\|
  \le \frac{1+\alpha_2}{r\sqrt{\alpha_1}}\|x^{(\bar\ell-1)}-x^{(\bar\ell)}\|_a \\
  &\le \frac{1+\alpha_2}{r\sqrt{\alpha_1}}\delta \le \frac{\varepsilon}{r}.
\end{aligned}
\]
Now, for any $x\in X$, we have that
\begin{align*}
\phi^{(r)}_z(x^{(\bar\ell)}) &\le \phi^{(r)}_z(x) + \frac{1}{r}\inner{(A-I)(x^{(\bar\ell-1)}-x^{(\bar\ell)})}{x-x^{(\bar\ell)}}\\
&\le \phi^{(r)}_z(x) + \frac{1}{r}\|(A-I)(x^{(\bar\ell-1)}-x^{(\bar\ell)})\| \|x-x^{(\bar\ell)}\| \\
&\le \phi^{(r)}_z(x) + \frac{\varepsilon}{r} \|x-x^{(\bar\ell)}\|,
\end{align*}
where we first applied \Cref{lem:M.weight.1}, then the Cauchy-Schwarz inequality, and finally the previous $\varepsilon$ bound.  The result then follows from \Cref{lem:delta-prox-char}.
\end{proof}

\section{Numerical Examples}
\label{sec:numerics}
We conduct two numerical examples: the optimal control of Burgers' equation with nontrivially weighted inner products and a density-based topology optimization example with $\phi$ of the form \eqref{eq:struc.nonsmooth}.
Throughout the numerical results, we use $\Delta_1 = 50$, $\eta_1 = 0.05$, $\eta_2 = 0.9$, $\gamma_1 = \gamma_2 = 0.25$, and $\gamma_3 = 2.5$.
For both examples, we stop \Cref{alg:cap} when
\begin{equation}\label{eq:num.stop}
  \psi_k(r_0) \le 10^{-5}.
\end{equation}
All studies were performed on a MacBook Pro running Sequoia 15.8
with 36 GB of memory and an M4 Max chip. 

\subsection{Optimal Control of Burgers' Equation}

We apply \Cref{alg:cap} to solve a discretized version of the following optimal control problem governed by Burgers' equation:
\begin{equation}
    \min_{z\in L^2(0,1)} \int_0^1([S(z)]-w)^2(x)\dx + \frac{\beta_1}{2}\int_0^1 z^2(x)\dx + \beta_2\int_0^1 |z|(x)\dx
\end{equation}
where $(0,1)$ is the physical domain, $\beta_1=10^{-4}$ and $\beta_2=10^{-2}$ are penalty parameters, $w(x)=-x^{2}$ is the target state, and $S(z)=u\in H^{1}(0,1)$ solves the weak form of Burgers' equation
\begin{eqnarray}
\label{eq:burgers}
    -\nu u'' + uu' = z+q \quad \text{in } (0,1), \quad 
    u(0)=0, \quad u(1)=-1.
\end{eqnarray}
Here, the viscosity is $\nu=0.08$ and the source is $q(x)=2(\nu+x^3)$. 
We discretize the state $u$ and $z$ using continuous piecewise linear finite elements on a uniform mesh with $n=512$ intervals ($h=1/512$).  
To compute $S(z)$, we solve the discretized Burgers' equation using Newton's method globalized with a backtracking line search.  
We exit the Newton iteration when the relative residual falls below $10^{-4}\sqrt{\epsilon_{\rm mach}}$, 
where $\epsilon_{\rm mach}$ is machine epsilon.
To approximately solve the trust-region subproblem \eqref{eq:tr-sub}, we employ the {\tt NCG} algorithm from \cite[Algorithm~4]{baraldi2024efficient}.

For the discretized problem, $X=\R^n$ endowed with the weighted inner product
\[
  \inp{x}{y}=\inp{x}{y}_M \coloneqq  x^\top M y = \sum_{i=1}^n\sum_{j=1}^n m_{i,j}x_i y_j,
\]
where $M\in\R^{n\times n}$ is the non-diagonal symmetric positive-definite (SPD) mass matrix.
This inner product represents a discretization of the infinite-dimensional $L^2(0,1)$ inner product, where $m_{i,j}$ is the $L^2(0,1)$ inner product of the $i$-th and $j$-th finite-element basis functions.
The fact that $M$ is not diagonal destroys separability for many $\phi$ like $\|\cdot\|_1$, complicating the evaluation of their proximity operators.
For this example, we leverage the framework described in \Cref{ss:weighted.prod}.
In particular, we set
\(
  a(x,y) = \inp{x}{y}_D \coloneqq  x^\top D y = \sum_{i=1}^n d_i x_i y_i,
\)
where $D=\textup{diag}(d)\in\R^{n\times n}$ is a positive diagonal matrix produced by lumping the mass matrix $M$.
The associated $A$-operator from \Cref{ss:weighted.prod} is given by $A = M^{-1}D$.
Moreover, the explicit forms of $M$ and $D$ are
\[
M=\frac{h}{6}\begin{pmatrix}
4 & 1 & \ldots & 0 & 0\\
1 & 4 & \ldots & 0 & 0\\
\vdots & \vdots & \ddots & \vdots &\vdots\\
0 & 0 & \ldots & 4 & 1\\
0 & 0 & \ldots & 1 & 4\\
\end{pmatrix} \in \mathbb{R}^{n\times n}\quad \text{and}\quad D=\frac{h}{6}\begin{pmatrix}
5 & 0 & \ldots & 0 & 0\\
0 & 6 & \ldots & 0 & 0 \\
\vdots & \vdots & \ddots & \vdots & \vdots\\
0 & 0 & \ldots & 6 & 0 \\
0 & 0 & \ldots & 0 & 5 \\
\end{pmatrix}\in \mathbb{R}^{n\times n},
\]
which yield the constants $\alpha_1=1$ and $\alpha_2=3$.
Furthermore, we approximate the $L^1(0,1)$-norm by the quantity
\[
  \beta_2\int_\Omega |z|(x)\dx \approx \phi(z)=\beta_2 h(\tfrac{5}{6} |z_1| + |z_2| + \ldots + |z_{n-1}| + \tfrac{5}{6}|z_n|),
\]
for which the $A$-weighted proximity operator is the usual soft-thresholding operator
\(
  \Proxrphi^a(z) = \textup{sign}(z)\odot\max\{|z|- \beta_2 r,0\}.
\)

In \Cref{tbl:1}, we summarize the performance of \Cref{alg:cap} for different values of the
inexact proximal gradient tolerance parameter $\kappa_{\textup{stat}}$ from \Cref{assump:inexact}. 
As $\kappa_{\textup{stat}}$ decreases, we require additional accuracy from the approximate proximity operator computed using \Cref{alg:inex.M.prox}.  
To generate these results, we cap the number of {\tt NCG} subproblem solver iterations to 15,
use the absolute and relative tolerances $\tau_{\textup{abs}}=10^{-5}$ and $\tau_{\textup{rel}}=10^{-3}$, respectively, and set
the spectral safeguards to $\lambda_{\min} = 10^{-12}$ and $\lambda_{\max}=10^{12}$.
The columns in \Cref{tbl:1} correspond to the value of $\kappa_{\textup{stat}}$, the wallclock time in seconds ({\tt time (s)}), the number of trust-region iterations ({\tt iter}), the number of evaluations of $f$ ({\tt nobj}), the number of evaluations of $\nabla f$ ({\tt ngrad}), the number of applications of the Hessian $\nabla^2 f$ ({\tt nhess}), the number of proximity operator evaluations ({\tt nprox}), and the average number of iterations of \Cref{alg:inex.M.prox} ({\tt av-piter}).
From \Cref{tbl:1}, we notice that different values of $\kappa_{\textup{stat}}$ have little effect on the performance of \Cref{alg:cap}, but have a drastic effect on the average number of iterations of \Cref{alg:inex.M.prox}.
For larger problems, this suggests that \Cref{alg:cap} can achieve significant savings by relaxing the required accuracy of the proximity operator evaluations.
\begin{table}[!ht]
\centering
{\ttfamily
\begin{tabular}{l r r r r r r r}
$\kappa_{\textup{stat}}$ & {\tt time (s)} & {\tt iter} & {\tt nobj} & {\tt ngrad} & {\tt nhess} & {\tt nprox} & {\tt av-piter} \\
   \hline
      {\tt 1e2  } & {\tt 0.21049}   &  {\tt 13 }  &  27  &    14 &     93  &   182 &         7.5385 \\
      {\tt 1e1  } & {\tt 0.14522}   &  {\tt 13 } &   27  &    14 &     93  &   182 &       7.5385 \\
      {\tt 1e0  } & {\tt 0.15767}    & {\tt 14 }  &  29   &   15  &   109   &  215  &       8.7143 \\
      {\tt 1e-1 } & {\tt 0.14283}   &  {\tt 14 } &   29  &    15 &    109  &   215 &             16 \\
      {\tt 1e-2 } & {\tt 0.13567}   &  {\tt 12 } &   25  &    13 &     77  &   149 &          17.75 \\
     {\tt 1e-3  } & {\tt 0.12378}    & {\tt 12 }  &  25   &   13  &    77   &  149  &       30.667 \\
     {\tt 1e-4  } & {\tt 0.14684}    & {\tt 14 } &   29   &   15  &   109   &  215  &           68 \\
     {\tt 1e-5 }  & {\tt 0.13407}    & {\tt 12 } & 25    &  13   &   77    & 149   &   84.417 \\
 \hline
\end{tabular}
}
\caption{Results for the optimal control of Burgers' equation with $\kappa_{\textup{stat}}\in\{10^{i}\,\vert\,i=-5,-4,\ldots,2\}$.
We report the wallclock time, 
\Cref{alg:cap} iterations, 
$f$  evaluations, $\nabla f$ evaluations, 
Hessian applications, 
proximity operator evaluations, and the average number of \Cref{alg:inex.M.prox} iterations.
}
\label{tbl:1}
\end{table}

In our final experiment, we incorporate inexact PDE solves in addition to inexact proximity operator evaluations,
terminating Newton's method for solving the discretized Burgers' equation \eqref{eq:burgers} when the relative residual falls below
\(
  \min\{10^{-2},\tau\},
\)
for $\tau$ guided by \Cref{alg:cap} as in \Cref{app:inexact}.
We choose $\kgrad=\kappa_{\textup{stat}}=1$ and $\kobj=10^3$ to balance the sizes of the initial tolerances passed to the PDE solver.
Since the dominant cost of Newton's method is the linear system solve at each iteration, we compare the average number of linear system solves per trust-region iteration for \Cref{alg:cap} using accurate and inexact PDE solves.
In particular, \Cref{alg:cap} averaged 5.3125 linear system solves per iteration when using inexact PDE solves, compared with 7.7222 when using exact PDE solves.

\subsection{Topology Optimization using the Total Variation}
The goal of topology optimization is to determine the distribution of a material within a domain $\Omega$ that minimizes, e.g., the compliance of the resulting structure, while satisfying a volume constraint.
Although uncommon in practice, one can promote sharp interfaces within the computed material density using total variation (TV) regularization \cite{sigmund2013topology,borrvall2001tv}.
The typical compliance minimization problem with TV regularization is
\begin{subequations}\label{eq:to-opt}
\begin{align}
  &\min_{\rho\in L^2(\Omega)} \; \int_{\Gamma_t} T(x)\cdot [S(\rho)](x)\,\mathrm{d}x + \beta_{\textup{TV}}\textup{TV}(\rho) \label{eq:to-opt-1} \\
  &\mathrm{subject\;to} \quad \int_\Omega \rho(x)\,\mathrm{d}x = v\vert\Omega\vert,\quad 0\le\rho\le 1\;\;\text{a.e.} \label{eq:to-opt-2}
\end{align}
\end{subequations}
For our results, we leverage the classical MBB example and implementation from \cite{andreassen2011efficient}, in which $\Omega=(0,150)\times(0,50)$ is the physical domain, $T$ is a point load of $-1$ in the vertical direction applied at $(0,50)$, $v=0.4$ is the volume
fraction,
\[
  \textup{TV}(\rho)\coloneqq \sup\left\{\int_\Omega \rho \,\textup{div}\,q\,\dx\,\Big\vert\, q\in C_c^1(\Omega,\R^2),\;\|q\|_{L^\infty(\Omega)} \le 1\right\},
\]
and $S(\rho)=u\in H^1(\Omega)^2$ solves the weak form of the linear elasticity equations
\begin{subequations}\label{eq:to-pde}
\begin{align}
  -\nabla\cdot(K(\rho):\varepsilon) &=0 &&\text{in $\Omega$} \\
  \varepsilon &= \tfrac{1}{2}(\nabla u + \nabla u^\top) &&\text{in $\Omega$}
\end{align}
\end{subequations}
Here,
\(
  K(\rho) \coloneqq [\kappa_{\min} + (1-\kappa_{\min})\mathbb{F}(\rho)^3] K_0,
\)
$K_0$ the standard isotropic elasticity matrix with Young's modulus 200 and Poisson ratio 0.29, $\mathbb{F}$ is the Helmholtz filter with radius 0.1 \cite{lazarov2011filters}, and $\kappa_{\min}=10^{-4}$.
Our boundary conditions and discretization are the same as \cite{andreassen2011efficient}.
We discretize the displacements $u$ and the filtered density $\mathbb{F}(\rho)$ using continuous piecewise linear finite elements on a $150\times 50$ uniform
quadrilateral mesh and the density $\rho$ using piecewise constants on the same mesh, resulting in $n=7500$ density degrees of freedom.
After discretization, the nonsmooth term is
\[
  \phi(\rho)=\left\{\begin{array}{ll}
    \beta_{\textup{TV}}\sum_{i=1}^n\|[D\rho]_i\|_2 & \text{if $\rho\in[0,1]^n$ and $\sum_{i=1}^n \rho_i = v|\Omega|$} \\
    +\infty &\text{otherwise}
    \end{array}\right.
\]
and $D$ is the differentiation matrix; cf.~\cite[Equation 15]{borrvall2001tv}.
Note that $\phi$ has the form \eqref{eq:struc.nonsmooth} and consequently, we will leverage \Cref{alg:struct.prim.dual} to approximate the proximity operator.

We compare the solutions to \eqref{eq:to-opt} with $\beta_{\textup{TV}}=0$ and $\beta_{\textup{TV}}=10^{-4}$.
For each experiment we employ the feasible initial guess $\rho\equiv v$.
To compute trial iterates in \Cref{alg:cap}, we employ the {\tt SPG2} subproblem solver with a maximum of 20 iterations, spectral safeguards $\lambda_{\min} = 10^{-4}$
and $\lambda_{\max}=10^{4}$, and absolute and relative tolerances $\tau_{\textup{abs}}=10^{-5}$ and $\tau_{\textup{rel}}=10^{-3}$, respectively.
We run experiments with $\beta_{\textup{TV}} = 0$, 
$\beta_{\textup{TV}} = 10^{-4}$ with inexact proximity operator evaluations,
and $\beta_{\textup{TV}} = 10^{-4}$ with high-accuracy proximity operator evaluations.
For both $\beta_{\textup{TV}}=10^{-4}$ experiments, we evaluate the proximity operator using \Cref{alg:struct.prim.dual}.
In the high-accuracy case, we exit \Cref{alg:struct.prim.dual} based on the primal-dual gap with $\varepsilon=10^{-6}$.
Attempts to drive it down further only increase the computational time without affecting the quality of the result.
In contrast, for the inexact proximity operator evaluations, we let \Cref{alg:cap} determine the $\varepsilon$ for the primal-dual gap and we add an additional stopping condition to exit the iteration if \eqref{eq:decr.inexact} is satisfied with $\kappa_{\textup{desc}}=0.75$.
In the context of \crefpart{assump:inexact}{assump:inexact.Psi}, we set $\kappa_{\rm stat} = 10^3$.
We further note that larger values of $\beta_{\textup{TV}}$ produced overly blurry designs, while smaller values had little effect on the final design compared to $\beta_{\textup{TV}}=10^{-4}$. 
\Cref{fig:tv} depicts the final designs for all three runs as well as the difference between the $\beta_{\textup{TV}}=10^{-4}$ run using inexact proximity operator evaluations and the $\beta_{\textup{TV}}=0$ run.
Note the improved inside corners for the $\beta_{\textup{TV}}=10^{-4}$ designs, resulting in a more intuitive and practical design. 
\begin{figure}[!htb]
    \centering
    \subfloat[$\beta_{\textup{TV}}=0$\label{fig:notv}]{\includegraphics[width=0.48\linewidth]{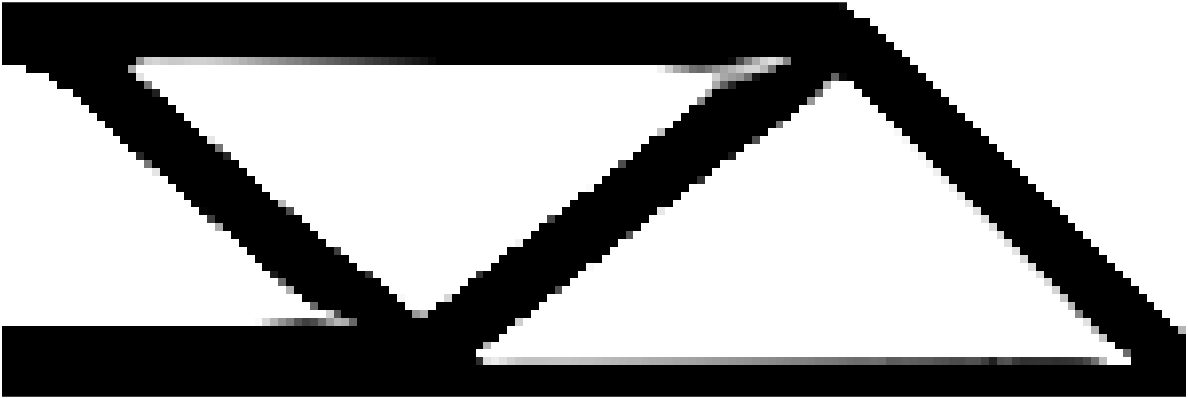}}
    \subfloat[$\beta_{\textup{TV}}=10^{-4}$, inexact]{\includegraphics[width=0.48\linewidth]{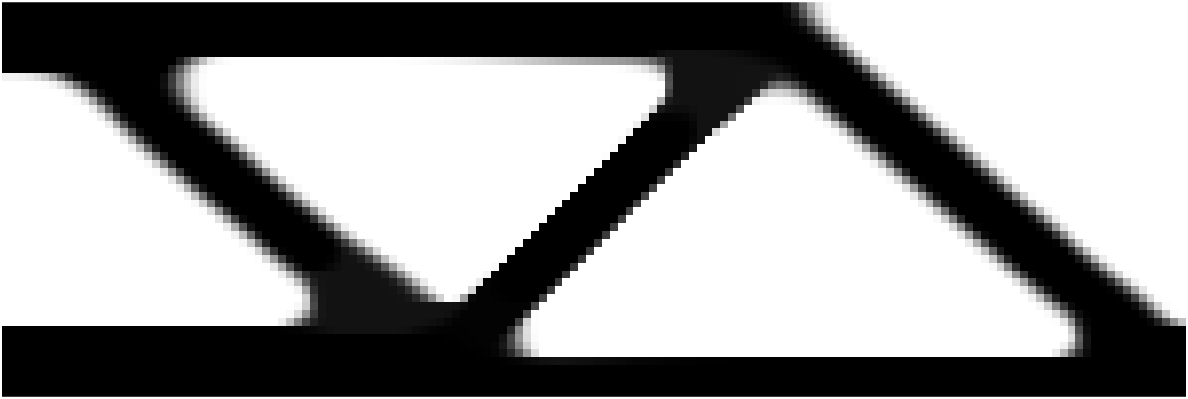}}\\
    \subfloat[$\beta_{\textup{TV}}=10^{-4}$, high accuracy]{\includegraphics[width=0.48\linewidth]{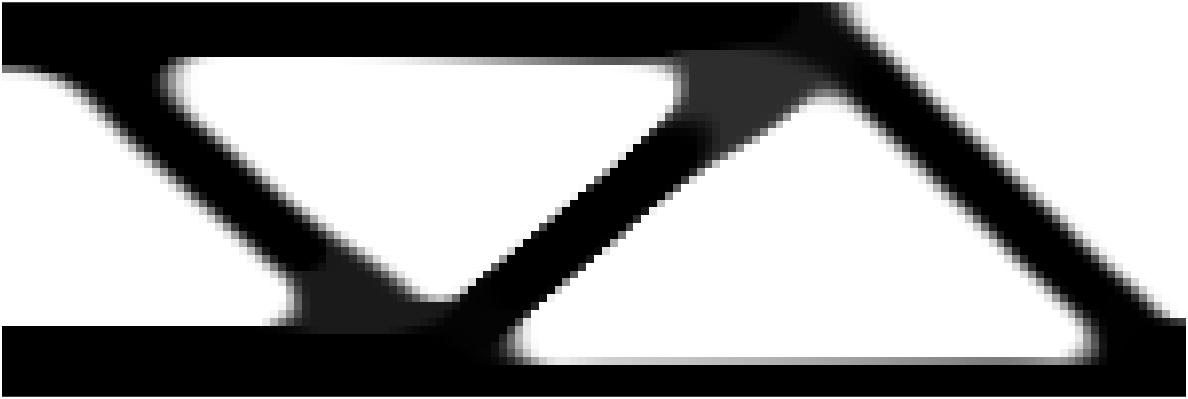}}
    \subfloat[Difference between $\beta_{\textup{TV}}=0$ case and inexact TV]{\includegraphics[width=0.48\linewidth]{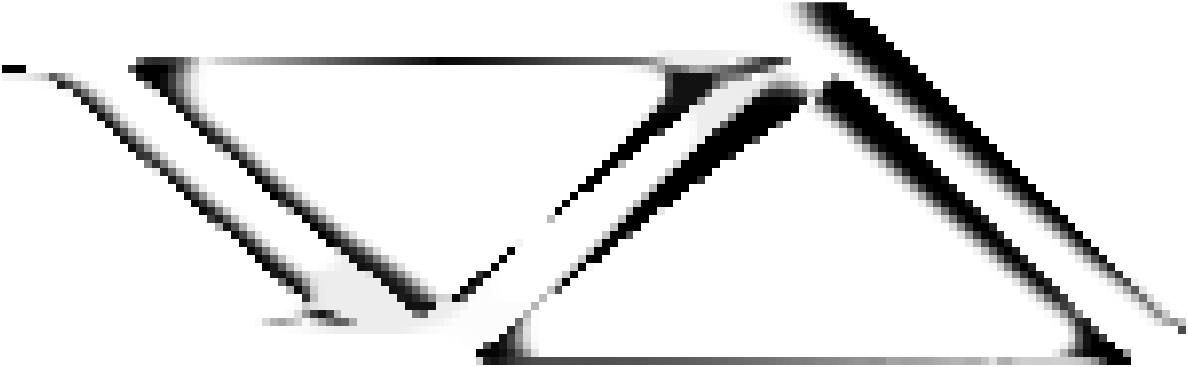}}
    \caption{Final design for $\beta_{\textup{TV}} = 0$ and $\beta_{\textup{TV}} = 10^{-4}$ (using inexact and high-accuracy proximity operator evaluations) and the difference between the $\beta_{\textup{TV}}=0$ and inexact $\beta_{\textup{TV}}=10^{-4}$ cases.}
    \label{fig:tv}
\end{figure}
\Cref{tab:exactstudy} demonstrates the performance of \Cref{alg:cap}.
\begin{table}[ht]
    \centering
    {\ttfamily
    \begin{tabular}{r|rrrrrrrr}
                    & time (s)    &    iter    & nobj   &   ngrad   &   nhess   &   nprox  & av-piter\\ \hline
      No-TV         & 9.94       & 42         & 43      &  43      &  924       &  967  & 0 \\
      inexact-TV    & 30.40      & 23         & 24      &  24      &  471       &  516 & 345\\
      accurate-TV   & 810.13     & 21         & 22      &  22      &  462       &  484  &  13057
    \end{tabular}
    }
    \caption{Algorithmic performance for topology optimization. The column {\tt time (s)} reports the wallclock time in seconds,  {\tt iter} trust-region iterations, {\tt nobj}, {\tt  ngrad} and {\tt nhess}
    function, gradient and Hessian evaluations, {\tt nprox} proximity operator
    evaluations, {\tt av-piter} the average number of \Cref{alg:struct.prim.dual} iterations per trust-region iteration to evaluate the proximity operator.}
    \label{tab:exactstudy}
\end{table}
The $\beta_{\textup{TV}}=0$ run converged in 42 iterations while $\beta_{\textup{TV}}=10^{-4}$ with inexact proximity operator evaluations converged in 23.
The proximity operator tolerance requested by the trust-region algorithm becomes quite tight as $\Psi_k$ decreases, requiring many iterations of \Cref{alg:struct.prim.dual} and resulting in a slower wallclock time --- approximately 30 seconds.
However, solving the TV subproblem to a tight tolerance, even $10^{-6}$ or tighter, takes upwards of 13 minutes despite the lower number of trust-region iterations. 
We leave the tasks of improving the computational efficiency for TV-regularized topology optimization and testing different approximations of the total variation as future work.

\appendix

\section{Inexact Gradient and Stationarity Metric}\label{app:inexact}

We now describe how to satisfy the inexact gradient and stationarity metric conditions in \Cref{assump:inexact}.
For this, we work in the same setting as \cite[Appendix~B]{Baraldi2023}. In particular, we assume that we have access to a gradient approximation $\hat{g}:X\times\mathbb{R}_+\to X$ that satisfies:
there exists $C_{\rm grad}\ge 0$ such that for any $x\in X$ and $\tau\ge0$, we have that
\begin{equation}\label{eq:inexactgrad}
  \|\nabla f(x)-\hat{g}(x,\tau)\| \le C_{\rm grad}\tau.
\end{equation}
\Cref{alg:inexactgrad} describes a procedure for satisfying \crefpart{assump:inexact}{assump:inexact.gradient} and \crefpart{assump:inexact}{assump:inexact.Psi} in finite time, which we prove in the next result.

\begin{algorithm}[!ht]
\caption{Inexact Gradient and Stationarity Metric Computation}
\label{alg:inexactgrad}
\begin{algorithmic}[1]
\Require{User-specified constants $\hat{\kappa}_{\textup{grad}},\,\hat{\kappa}_{\textup{stat}}\in\R_{++}$, user-specified tolerance $\hat{\mu}\in(0,1)$, user-specified parameters $\eta_{\textup{stat}}\in(0,1)$, $\zeta_{\textup{stat}}>1$ and $\xi_k\in\R_{++}$, the current iterate $x_k\in X$, step length $r_0\in\R_{++}$ and the trust-region radius $\Delta_k$}
\State{Set $\tau_k^-\gets\Delta_k$,
$\tau_k^{\textup{grad}}\gets\hat\kappa_{\textup{grad}}\tau_k^-$ and
$\tau_k^{\textup{stat}}\gets r_0\hat\kappa_{\textup{stat}}[\eta_{\textup{stat}}\min\{\hat\mu\tau_k^-,\xi_k\}]^{\zeta_{\textup{stat}}}$}
\State{Compute $g_k \gets \hat{g}(x_k,\tau_k^{\textup{grad}})$}
\State{Compute $x_k(r_0)\approx_1 \Prox_{r_0\phi}(x_k-r_0 g_k)$ with $\tau_k^{\textup{stat}}$-precision}
\State{Set $\tau_k^+\gets\min\{\psi_k(r_0),\Delta_k\}$}
\While{$\tau_k^+ < \hat\mu\tau_k^-$}
  \State{Set $\tau_k^-\gets\tau_k^+$, $\tau_k^{\textup{grad}}\gets\hat\kappa_{\textup{grad}}\tau_k^-$ and
$\tau_k^{\textup{stat}}\gets r_0\hat\kappa_{\textup{stat}}[\eta_{\textup{stat}}\min\{\hat\mu\tau_k^-,\xi_k\}]^{\zeta_{\textup{stat}}}$}
  \State{Compute $g_k \gets \hat{g}(x_k,\tau_k^{\textup{grad}})$}
  \State{Compute $x_k(r_0)\approx_1 \Prox_{r_0\phi}(x_k-r_0 g_k)$ with $\tau_k^{\textup{stat}}$-precision}
  \State{Set $\tau_k^+\gets\min\{\psi_k(r_0),\Delta_k\}$}
\EndWhile
\end{algorithmic}
\end{algorithm}

\begin{proposition}
  Let $x_k$ be an iterate of \Cref{alg:cap} for which
  \begin{equation}\label{eq:app.stat}
    \vartheta_k\coloneqq\frac{1}{r_0}\|\Prox_{r_0\phi}(x_k-r_0\nabla f(x_k))-x_k\| > 0.
  \end{equation}
  Then \Cref{alg:inexactgrad} terminates in finitely many iterations with $g_k$ and $\psi_k(r_0)$ satisfying \Cref{assump:inexact}
  with $\kappa_{\textup{grad}}=\hat\mu^{-1}\hat\kappa_{\textup{grad}}C_{\textup{grad}}$ and $\kappa_{\textup{stat}}=\hat\kappa_{\textup{stat}}$.
\end{proposition}
\begin{proof}
Suppose that \Cref{alg:inexactgrad} produces infinitely many iterations.
Let $\{\tau_k^\ell\}$ denote the sequence of tolerances generated by \Cref{alg:inexactgrad} so that at the $\ell$-th iteration we have
\[
  \tau_k^{\ell+1} < \hat\mu\tau_k^\ell < \tau_k^\ell.
\]
Consequently $\{\tau_k^\ell\}$ is a nonnegative decreasing sequence and therefore has a limit $\tau^\star\ge 0$.
Moreover, because of \eqref{eq:app.stat}, we have that $\tau^\star > 0$.
Otherwise, $\tau_k^\ell\to 0$ which implies that $\tau_k^{\textup{grad}}\to 0$ and $\tau_k^{\textup{stat}}\to 0$ as \Cref{alg:inexactgrad} iterates.
Therefore, the sequence of approximations of $g_k$ converges to $\nabla f(x_k)$ by \eqref{eq:inexactgrad} and the sequence of approximations of $\psi_k(r_0)$ converges to $\vartheta_k$ by \Cref{lem:inexac.exact.prox.diff} combined with nonexpansivity of the proximity operator.
However, $\vartheta_k>0$, $\Delta_k>0$ and therefore, $\min\{\psi_k(r_0),\Delta_k\}$ converges to the positive limit $\min\{\vartheta_k,\Delta_k\}$, leading to a contradiction since $\tau_k^{\ell+1}=\min\{\psi_k(r_0),\Delta_k\}$.
Now, we notice that since $\tau_k^\ell\to\tau^\star$, we have that $\hat\mu\tau_k^\ell\to\hat\mu\tau^\star < \tau^\star$ and so for any $\varepsilon\in(0,(1-\hat\mu)\tau^\star]$, there exist $\ell_\varepsilon>0$ such that
\[
  \hat\mu\tau_k^\ell \le \hat\mu\tau^\star + \varepsilon \le \tau^\star \le \tau_k^{\ell+1} \quad\forall\,\ell\ge \ell_\varepsilon.
\]
Hence, \Cref{alg:inexactgrad} cannot produce an infinite sequence as was to be shown.
To finish, suppose \Cref{alg:inexactgrad} exits at iteration $\ell$ returning $\tau_k^{\ell}$ that satisfies
\[
  \tau_k^\ell \le \hat\mu^{-1}\tau_k^{\ell+1} = \hat\mu^{-1}\min\{\psi_k(r_0),\Delta_k\}.
\]
The satisfaction of \crefpart{assump:inexact}{assump:inexact.gradient} follows directly from this and the approximation condition \eqref{eq:inexactgrad}.
On the other hand, the satisfaction of \crefpart{assump:inexact}{assump:inexact.Psi}
follows since the last $\tau_k^{\textup{stat}}$ used to approximate $\psi_k(r_0)$ satisfies
\[
  \tau_k^{\textup{stat}}
  = r_0\hat\kappa_{\textup{stat}}[\eta_{\textup{stat}}\min\{\hat\mu\tau_k^\ell,\xi_k\}]^{\zeta_{\textup{stat}}}
  \le r_0\hat\kappa_{\textup{stat}}[\eta_{\textup{stat}}\min\{\psi_k(r_0),\Delta_k,\xi_k\}]^{\zeta_{\textup{stat}}}.
\]
The result then follows from \Cref{lem:inexac.exact.prox.diff}.
\end{proof}

\bibliographystyle{siamplain}
\bibliography{references}

\end{document}